\documentclass[twoside,11pt]{article}

\usepackage[preprint]{jmlr2e}

\usepackage{microtype}
\usepackage{graphicx}
\usepackage{bm}
\usepackage{algorithm,algpseudocode}
\usepackage{multirow}
\usepackage{amsmath}
\usepackage{subfigure}
\usepackage{amssymb}
\usepackage{booktabs} 
\usepackage{color}
\usepackage{balance}
\usepackage{float}
\usepackage{subeqnarray}
\usepackage{cases}
\usepackage{tabularx}
\usepackage{threeparttable}
\usepackage{makecell}
\usepackage{natbib}
\usepackage{hyperref}

\newtheorem{assumption}{Assumption}

\jmlrheading{1}{2000}{1-48}{4/00}{10/00}{}{}

\ShortHeadings{Moreau Envelope based Reformulation and Algorithm for Bilevel Optimization}{}
\firstpageno{1}

\begin{document}
	
	\title{ Moreau Envelope Based Difference-of-weakly-Convex Reformulation and Algorithm for Bilevel Programs}
	
		\author{\name Lucy Gao \email lucy.gao@ubc.ca \\
		\addr Department of Statistics, University of British Columbia
		\AND
		\name Jane J. Ye \email janeye@uvic.ca \\
		\addr Department of Mathematics and Statistics, University of Victoria
		\AND 
		\name  Haian Yin  \email yinha@sustech.edu.cn \\
		\addr Department of Mathematics, Southern University of Science
		and Technology\\ National Center
		for Applied Mathematics Shenzhen
		\AND 
		\name  Shangzhi Zeng  \email zengshangzhi@uvic.ca \\
		\addr Department of Mathematics and Statistics, University of Victoria
		\AND 
		\name Jin Zhang  \email zhangj9@sustech.edu.cn\\
		\addr Corresponding author\\ Department of Mathematics, Southern University of Science
		and Technology\\ National Center
		for Applied Mathematics Shenzhen\\  Peng Cheng Laboratory
	}
	

	\editor{ }

\maketitle
	
	\begin{abstract}

	Bilevel programming has emerged as a valuable tool for hyperparameter selection, a central concern in machine learning. In a recent study by Ye et al. (2023), a value function-based difference of convex algorithm was introduced to address bilevel programs. This approach proves particularly powerful when dealing with scenarios where the lower-level problem exhibits convexity in both the upper-level and lower-level variables. Examples of such scenarios include support vector machines and $\ell_1$ and $\ell_2$ regularized regression. In this paper, we significantly expand the range of applications, now requiring convexity only in the lower-level variables of the lower-level program. We present an innovative single-level difference of weakly convex reformulation based on the Moreau envelope of the lower-level problem. We further develop a sequentially convergent Inexact Proximal Difference of Weakly Convex Algorithm (iP-DwCA). To evaluate the effectiveness of the proposed iP-DwCA, we conduct numerical experiments focused on tuning hyperparameters for kernel support vector machines on simulated data.
	
 \end{abstract}
	

\section{Introduction}
A bilevel program is an optimization problem over a pair of optimization variables $(x, y)$ where $y$ is constrained to be a solution of a lower level program indexed by the value of $x$. For example, consider the following bilevel program: 
	\begin{equation*}\label{BP}
		{\rm (BP)}~~~~~~~~~~\begin{aligned}
			\min_{x\in X, y\in Y}  ~~& F(x,y) \\
			s.t. ~~~~&  y \in S(x),
		\end{aligned}
	\end{equation*}
	where $S(x) := {\arg \min}_{y \in Y} ~ \{ f(x, y) \text{ s.t. } g(x, y) \leq 0\}$ is the set of optimal solutions for the lower level program $(P_x)$:
	\[
	\begin{aligned}
		(P_x): \quad \min_{y \in Y}~&f(x,y) \\s.t.~& g(x,y) \le 0.
	\end{aligned}
	\] 
 We assume throughout this paper that $S(x)$ is nonempty for each $x\in X$.  

A typical first step in solving bilevel program is to reformulate them as a single level optimization problem. One such reformulation is the value function reformulation, which would rewrite (BP) as:  
	\begin{equation*}\label{VPnew}
		{\rm (VP)}~~~~~~~~~~\begin{aligned}
			\min_{x\in X,y\in Y}  ~~& F(x,y) \\
			s.t. ~~~~& f(x,y)-v(x)\leq 0,
			\\ & g(x,y)\leq 0,
		\end{aligned}
	\end{equation*} 	
	where $v(x):=\inf_{w\in Y}\left \{ f(x,w) \mbox{ s.t. } g(x, w)\leq 0 \right \}$
	 is the value function of the lower level program $(P_x)$. This is equivalent to $(BP)$ because by definition, $v(x) \leq f(x, y)$ for all $(x, y) \in X \times Y$ satisfying $g(x,y)\leq 0$, and $v(x) = f(x, y)$ implies that $y \in S(x)$. 
  
  This is typically a non-convex and non-smooth optimization problem, even when all of the defining functions are smooth and convex. However, 
   \citet{ye2021difference} observe that under appropriately chosen assumptions, (VP) is a difference of convex (DC) program -- i.e. each of the objective functions and constraint functions are the difference of two convex functions. The key assumptions of \cite{ye2021difference} are that: 
   (i) 
   $X$ and $Y$ are convex and (ii) $f(x, y)$ and $g(x, y)$ are convex in \emph{both} $x$ and $y$. These assumptions are needed to ensure that the value function is convex. \citet{ye2021difference} then propose a DC programming algorithm for solving (VP).

An important application of bilevel programming appears in the context of predictive modelling, where building an accurate model typically requires tuning a set of hyperparameters ($\lambda)$. Two popular general proposals for tuning hyperparameters involve attempting to solve bilevel optimization problems: sample splitting and cross-validation. The idea is to minimize the validation or cross-validation error over the choice of hyperparameters $\lambda$. In this context, the training loss function is often a convex function of the model parameters for a fixed value of the hyperparameters, which leads to a convex lower level program. However, the training loss function is rarely convex when viewed as a function of the model parameters \emph{and} the hyperparameters. 

\citet{pmlr-v162-gao22j} carefully reformulate the bilevel programming problems induced by tuning the hyperparameters of several popular prediction algorithms (e.g. elastic net and support vector machines) to avoid violating the assumption that the defining functions of the lower level program are convex in both variables. This reformulation strategy yields good theoretical and empirical results, but relies heavily on the problem structure at hand, which makes it cumbersome to generalize the convergence analysis to prediction algorithms not explicitly considered in  \citet{pmlr-v162-gao22j}. Furthermore, the reformulation strategy of \citet{pmlr-v162-gao22j} does not naturally extend to tuning the hyperparameters of certain prediction algorithms such as kernel support vector machines. 

In this paper, we aim to solve (BP) without assuming that $f(x, y)$ is convex in both $x$ and $y$. To this end, we propose an alternative to the value function reformulation that only relies on the weaker assumption that $(P_x)$ is a convex program for any $x \in X$. The key idea behind our reformulation is to replace the value function in (VP) with a Moreau envelope function \citep{moreau1965proximite, rockafellar2009variational}: a smooth approximation to the objective function of the lower level program that maintains the same minimizers when the objective function is convex.  We show that this Moreau envelope function  is  \emph{weakly} convex and Lipschitz continuous under the mild additional assumption that $f(x, y)$ is weakly convex in both $x$ and $y$. Thus, under appropriate assumptions, our Moreau envelope reformulation is a difference of weakly convex program, which can be solved carefully applying a DC programming algorithm. 

The rest of the paper is organized as follows. We first introduce the Moreau envelope reformulation (Section \ref{sec:moreau-reform}) and its properties (Section \ref{sec:moreau-prop}). Then, we design an algorithm based on the Moreau envelope reformulation (Section \ref{sec:alg}).   In Section \ref{sec:theory}, we establish conditions under which the proposed algorithm converges toward ``good" quality solutions. We provide examples of hyperparameter tuning problems that fit into our algorithmic framework in Section \ref{sec:examples}. Numerical experiments on hyperparameter selection problems are in Section \ref{sec:exp}. 

\section{Moreau envelope reformulation of bilevel programs}
\label{sec:moreau-reform} 

\subsection{Moreau envelope function and its proximal mapping}
\label{sec:moreau-def}

The goal of this paper is to find a way to solve (BP) that mainly relies on the assumption that $(P_x)$ is a convex optimization problem for any $x \in X$:
\begin{assumption}
For any $x \in X$, $f(x, \cdot)$ and $g(x, \cdot)$ are convex functions defined on the closed convex set $Y$. 
\label{a1}
\end{assumption} 
Thus, we will assume that Assumption \ref{a1} holds throughout this subsection.

Let 
\begin{equation} 
\mathcal{F}(x) := \{y \in Y: g(x, y) \leq 0\} \quad \text{ and } \quad  C := \{(x, y) \in X \times Y: y \in \mathcal{F}(x) \} \label{defFC}
\end{equation} 
respectively represent the feasible solution set for the lower level program $(P_x)$ at a given value $x \in X$, and the set of $(x, y)$ pairs such that $x \in X$ and $y$ is feasible for $(P_x)$.  

Then, for any $x \in X$, $S(x)$ is the minima of the following non-smooth convex function: 
\begin{equation}\label{unconstrainedfunc}
				 y \rightarrow f(x,y)+\delta_C(x, y),
\end{equation} where $\delta_C(\cdot, \cdot)$ denotes the indicator function on the set $C$. For any $x \in X$, we can construct a smooth approximation of the function defined in \eqref{unconstrainedfunc} via its Moreau envelope \citep{moreau1965proximite}: 
\begin{equation} 
y \rightarrow  \inf_{w \in Y} \left\{ f(x,w) + \frac{1}{2\gamma}
			\|w - y\|^2 + \delta_C(x, w) \right \}. \label{funcnew} 
\end{equation} 
Replacing the $\inf$ in \eqref{funcnew} with $\arg \min$ yields the proximal mapping of \eqref{funcnew}. The objective function in \eqref{funcnew} is coercive and strongly convex (Assumption \ref{a1}), which implies that the minimum always exists.

We define  $v_{\gamma}: X \times \mathbb{R}^m \mapsto \mathbb{R}$ as follows: 
	\begin{equation}\label{def_v1}
		\begin{aligned}
			v_{\gamma}(x,y) & :
			= \inf_{w \in Y} \left\{ f(x,w) + \frac{1}{2\gamma}
			\|w - y\|^2 + \delta_C(x, w) \right\}.
		\end{aligned}
	\end{equation}
 That is, for any $x \in X$, $v_\gamma(x, \cdot)$ is the Moreau envelope of the function that $S(x)$ minimizes. 
We further define: 
	\begin{equation}\label{def_prox}
		\begin{aligned}
			S_{\gamma}(x,y) & :=\displaystyle  \mathrm{argmin}_ {w \in Y} \left\{ f(x,w) + \frac{1}{2\gamma}\|w - y\|^2 + \delta_C(x, w) \right\}.
		\end{aligned} 
	\end{equation}
 That is, for any $x \in X$, $S_\gamma(x, \cdot)$ is the proximal mapping of the function that $S(x)$ minimizes. 

 Chapter 1.G of \citep{rockafellar2009variational} provides a detailed review of Moreau envelopes and their associated proximal mappings. 

\subsection{Moreau envelope function reformulation of (BP)}

Our key idea is to replace the value function $v(x)$ in (VP) with the Moreau envelope function in \eqref{def_v1} as follows: 
\begin{equation*}\label{general_reformulated_problem}
		{\rm (VP)}_\gamma~~~~~~~~~~\begin{aligned}
			\min_{(x,y)\in C}  ~~& F(x,y) \\
			s.t. ~~~~& f(x,y)-v_\gamma (x,y)\leq 0.
		\end{aligned}
	\end{equation*} 

We first provide some intuition for why ${\rm (VP)}_\gamma$ is equivalent to $(BP)$ under Assumption \ref{a1} for finite values of $\gamma$. Recall from Section 1 that $(VP)$ is equivalent to $(BP)$. The Moreau envelope function $v_\gamma(x, y)$ in \eqref{def_v1} simply adds a term with gradient 0 to the function defined in \eqref{unconstrainedfunc}, i.e. the objective function of the value function $v(x)$, which preserves the stationary points. Under Assumption \ref{a1}, all global minima of the function defined in \eqref{unconstrainedfunc} are stationary points, so preserving the stationary points is equivalent to preserving the infimum. 

We formalize this intuition in the statement and proof of the following theorem. 
		\begin{theorem}\label{thm_reformulate}
				Let $\gamma
				>0$. Under Assumption 1, ${\rm (VP)}_\gamma$  is equivalent to ${\rm (BP)}$.
			\end{theorem}
			\begin{proof} It suffices to show that the feasible sets of $(VP)_\gamma$ and $(BP)$ are the same. 

	Let $(x,y)$ be any feasible point of ${\rm (VP)}_\gamma$. Then we have $(x,y) \in X \times Y$, $g(x,y) \le 0$ and 
				\[
				f(x,y) \le v_\gamma(x,y) := \inf_{w \in Y} \left\{ f(x,w) +\delta_C(x,w)+ \frac{1}{2\gamma}\|w - y\|^2 ~ 
				 \right\} \le f(x,y),
				\]
				and thus $f(x,y) = v_\gamma(x,y)$ and $y \in \mathrm{argmin}_{w \in Y}\left\{ f(x,w) +\delta_C(x,w)+ \frac{1}{2\gamma}\|w - y\|^2 
				 \right\} $. Taking the convex subdifferential with respect to $w$ of the function 
     \begin{equation}\label{func}
				 w \rightarrow f(x,w)+\delta_C(x,w)+\frac{1}{2\gamma}\|w-y\|^2
				 \end{equation} at $w=y$ via subdifferential calculus yields
				\begin{equation} 0 \in \partial_w ( f + \delta_{C})(x,y) , \label{0subgrad}\end{equation}
				 where $\partial_w $ denotes the partial subdifferential with respect to variable $w$. Assumption \ref{a1} implies that the function $w \mapsto f(x, w) + \delta_C(x, w)$ is convex. Thus, \eqref{0subgrad} implies that $y \in S(x)$. We have now shown that $(x, y)$ is a feasible point of $(BP)$. 
				
				Now, let $(x,y)$ be a feasible point of problem (BP). Then  we have $(x,y) \in X \times Y$, $y \in S(x)$ and thus the stationary condition $0 \in \partial_w ( f + \delta_{C})(x,y)$ must hold.
				It follows by subdifferential calculus that
		$0 \in \partial_w( f + \delta_{C}+h)(x,y),$
		where $		h(x,w):= \frac{1}{2\gamma}\|w-y\|^2$.
Finally, the convexity of the function defined in \eqref{func}, we have $y\in S_\gamma(x,y)$ and $f(x,y) = v_\gamma(x,y)$. Therefore  $(x,y)$ is an feasible point to problem ${\rm (VP)}_\gamma$.
			\end{proof}

As $\gamma \rightarrow \infty$, the proximal term in the Moreau envelope function vanishes. Thus, to unify the value function and Moreau envelope function reformulations under a single problem class, we define 
$v_\infty(x, y) := v(x)$ and $(VP)_\infty := (VP)$.

\section{Properties of the Moreau envelope function} 
\label{sec:moreau-prop}

If Assumption 1 holds and $f(\cdot, \cdot)$ and $g(\cdot, \cdot)$ are convex functions in both variables, then the value function $v(x)$ is convex and locally Lipschitz continuous {under some mild additional assumptions}; see Lemma 3, \citealt{ye2021difference}).  Furthermore, \citet{ye2021difference} are able to perform sensitivity analysis of the value function. These properties were critical to designing an algorithm to solve (VP). In this section, we will extend these results to the Moreau envelope function $v_\gamma(x, y)$ in order to design an algorithm to solve $(VP)_\gamma$. 

\subsection{Weak convexity} 

A function $\phi(z) $
is said to be weakly convex on  a convex set $U$  if there exists a modular $\rho\geq  0$ such that the function $\phi(z)+ \frac{\rho}{2}\|z\|^2$ is convex on $U$. Equivalently, $\phi(z)$ is weakly convex on $U$ if and only if there exists  $\rho\geq 0$ such that for any points $z,z'\in U$ and $\lambda \in [0,1]$, the approximate secant inequality holds:
			$$\phi(\lambda z+(1-\lambda) z') \leq \lambda \phi(z)+(1-\lambda) \phi(z') +\frac{\rho \lambda(1-\lambda)}{2} \|z-z'\|^2.$$
We also call a function $\rho$-weakly convex if it is weakly convex with modular equal to $\rho$. The class of  $\rho$-weakly convex functions includes convex functions as a special case when $\rho=0$.

The concept of weakly convex functions, first introduced in \cite{nurminskii1973quasigradient}, extends the class of convex functions to include a large class of nonconvex functions.  In fact, any  function with a bounded Lipschitz gradient  is weakly convex. More generally,  any function of the form
			 $\phi(z)=h(c(z)),$
			 with $h$  begin convex and Lipschitz, and $c$ being a smooth map with a bounded Lipschitz Jacobian, is  weakly convex \cite[Lemma 4.2]{drusvyatskiy2019efficiency}.

We will now establish conditions under which the Moreau envelope function is weakly convex. 

\begin{assumption} 
$X, Y$ are convex, $g(x, y)$ is convex on $X\times Y$, and $f(x, y)$ is $\rho_f$-weakly convex on $C$. \label{a2new}
\end{assumption} 

\begin{theorem}\label{Thm2}
Suppose that Assumptions \ref{a1} and \ref{a2new} hold. If $\rho_f>0$ and $\gamma \in (0, 1/\rho_f)$, then the proximal value function $v_\gamma(x, y)$ is $\rho_v$-weakly convex on $X\times Y$ for $\rho_v = \frac{\rho_f}{1- \gamma \rho_f}$.
\end{theorem} 
\begin{proof}
				First we extend the definition of proximal relaxed value function  from any element $(x,y)\in X\times \mathbb{R}^m$ to  the whole space $ \mathbb{R}^n\times \mathbb{R}^m$ as follows:
				$$v_{\gamma}(x,y) := \inf_{w \in \mathbb{R}^m}  \left\{ f(x,w) + \frac{1}{2\gamma}\|w - y\|^2 + \delta_C(x,w) \right\} \qquad \forall x\in \mathbb{R}^n, y \in \mathbb{R}^m.$$
				In order to show that  $v_{\gamma}(x,y)$ is $\rho_v$-weakly convex, it suffices to show that the convexity of the following function 
				\begin{equation}\label{phi-gamma}
				v_{\gamma}(x,y) + \frac{\rho_v}{2}\|(x,y)\|^2 = \inf_{w \in \mathbb{R}^m} 
    \phi_{\gamma, \rho_v}(x, y, w),
				\end{equation}  where 
				\[
				\begin{aligned}
				&	\phi_{\gamma, \rho_v}(x, y, w)  := \,f(x,w) + \frac{\rho_v}{2}\|(x,y)\|^2 + \frac{1}{2\gamma}\|w - y\|^2 + \delta_C(x,w)\\
    &=f(x,w) + \frac{\rho_f}{2}\|(x,w)\|^2+ \delta_C(x,w) + \frac{\rho_v - \rho_f}{2}\|x\|^2 +  \frac{1/\gamma-\rho_f}{2}\|w\|^2 + \frac{1/\gamma + \rho_v}{2}\|y\|^2  - \frac{1}{\gamma} \langle w,y\rangle.
				\end{aligned}\]   Since $f$ is $\rho_f$-weakly convex  on  the convex set $C$, the  function $$f(x,w) + \frac{\rho_f}{2}\|(x,w)\|^2 + \delta_C(x,w)$$ is convex with respect to $(x,w)$.  
				Next, when $\gamma \in (0, 1/\rho_f)$ and $\rho_v  \ge \frac{\rho_f}{1 - \gamma\rho_f} $, we have $\rho_v > \rho_f$, $(1/\gamma-\rho_f)(\rho_v+1/\gamma) \ge 1/\gamma^2$ and thus function 
				\begin{equation}\label{lem1_eq1}
					\frac{\rho_v - \rho_f}{2}\|x\|^2 +\frac{1/\gamma-\rho_f}{2}\|w\|^2 + \frac{1/\gamma + \rho_v}{2}\|y\|^2 - \frac{1}{\gamma} \langle w,y\rangle
				\end{equation}
				is convex with respect to $(x,y,w)$.
    Therefore, when $\gamma \in (0, 1/\rho_f)$ and $\rho_v  \ge \frac{\rho_f}{1 - \gamma\rho_f} $, the extended-valued function $\phi_{\gamma, \rho_v}(x, y, w)$ is convex with respect to $(x, y, w)$ on $\mathbb{R}^n \times \mathbb{R}^m \times \mathbb{R}^m$. 
				Then, the convexity of the function (\ref{phi-gamma})
follows from \cite[Theorem 1]{rockafellar1974conjugate}. Hence the function $v_{\gamma}(x,y) + \frac{\rho_v}{2}\|(x,y)\|^2$ restricted on set $X \times \mathbb{R}^m$ is convex and $v_{\gamma}(x,y)$ is $\rho_v$-weakly convex on $X \times \mathbb{R}^m$.  
   \end{proof}

\subsection{Lipschitz continuity}

To establish local Lipschitz continuity for the Moreau envelope function $v_\gamma(x, y)$, we will need to strengthen Assumption \ref{a2new} as follows. 

\begin{assumption} $X$ is convex, $g(x, y)$ is convex on ${\cal O} \times Y$ where  ${\cal O}$ is an open convex set containing set $ X$. $f$ is $\rho_f$-weakly convex on $D$ with $\rho_f>0$, where $D=\{(x,y)\in {\cal O} \times Y \mid g(x,y)\leq 0\}$, and ${\cal F}(x)\not =\emptyset $ for all $x\in {\cal O}$, where ${\cal F}(x)$ is defined in \eqref{defFC}.  \label{a3}
\end{assumption} 
\begin{theorem}\label{Thm3} Suppose that Assumptions \ref{a1} and \ref{a3} hold. If $\rho_f > 0$ and $\gamma \in (0, 1/\rho_f)$, then $v_\gamma(x,y)$ is Lipschitz continuous around any point in set $X\times Y$.
\end{theorem}
	\begin{proof} By assumption,   ${\cal F}(x)\not =\emptyset$  for all  $x$ in  the open set ${\cal O}$.  Since when $\gamma \in (0, 1/\rho_f)$ and $\rho_v  > \frac{\rho_f}{1 - \gamma\rho_f} $, the function given in \eqref{lem1_eq1} is strongly convex with respect to $(x,y,w)$ and thus $\phi_{\gamma, \rho_v}(x,y,w)$ is strongly convex with respect to $(x, y, w)$. Hence, when $\gamma \in (0, 1/\rho_f)$ and $\rho_v  > \frac{\rho_f}{1 - \gamma\rho_f} $, for any $x \in \mathcal{O}$ and ${\cal F}(x) \neq \emptyset$, $$v_{\gamma}(x,y) + \frac{\rho_v}{2}\|(x,y)\|^2 = \inf_{w \in \mathbb{R}^m}  \left\{ \phi_{\gamma, \rho_v}(x,y,w) \right\} \in \mathbb{R}$$ and thus
				the function $v_{\gamma}(x,y) + \frac{\rho_v}{2}\|(x,y)\|^2 $ is proper convex.
				Since   $\mathrm{dom}(v_{\gamma} + \frac{\rho_v}{2}\|\cdot\|^2)$ can be characterized as $$\{(x,y)~|~v_{\gamma}(x,y) + \frac{\rho_v}{2}\|(x,y)\|^2<+\infty\}= \{(x,y): {\cal F}(x)\not =\emptyset\}\supseteq {\cal O} \times \mathbb{R}^m\supseteq X \times \mathbb{R}^m$$  we have that $X \times \mathbb{R}^m \subseteq \mathrm{int}(\mathrm{dom}(v_{\gamma}(\cdot) + \frac{\rho_v}{2}\|\cdot\|^2))$. Then, the result on  local Lipschitz continuity of the function $v_{\gamma}(x,y) + \frac{\rho_v}{2}\|(x,y)\|^2$ follows from \cite[Theorem 10.4]{rockafellar1970convex}, and local Lipschitz continuity of the function $v_{\gamma}(x,y)$ follows.
			\end{proof}

\subsection{Sensitivity analysis of the Moreau envelope function}

For convenience in this section we convert a function $\phi(z)$ defined on a convex $U$ into an extended-valued function 
$\phi (z)+\delta_U(z)$. 
			  We recall  that  an extended-valued function $\varphi(z) : \mathbb{R}^d \rightarrow \mathbb{R} \cup\{\infty\}$ is a $\rho$-weakly convex function if $\varphi(z)+ \frac{\rho}{2}\|z\|^2$ is convex. 
Let $\varphi: \mathbb{R}^d \rightarrow \mathbb{R} \cup\{\infty\}$ be a proper closed weakly convex function. Then at any point $\bar{z} \in {\rm dom} \varphi:=\{z| \varphi( z)<\infty\}$, we define the Fr\'{e}chet (regular) subdifferential of $f$ at $\bar{z}$ as
\[
\partial \varphi(\bar{z}) = \{ v \in \mathbb{R}^d ~|~ \varphi(z)  \ge \varphi(\bar{z}) + \langle v, z - \bar{z}\rangle + o(\|z - \bar{z}\|) , \forall z \in \mathbb{R}^d\},
\] where $o(t)$ is a function satisfying $o(t)/t \rightarrow 0$ as $t\downarrow 0$.
If $\bar{z} \notin \mathrm{dom}\, \varphi$, we set $\partial \varphi(\bar{z}) = \emptyset$. 	
 For weakly convex functions, the Fr\'{e}chet (regular) subdifferential coincides with the limiting subdifferential and the Clarke subdifferential, see, e.g., \cite[Proposition 3.1 and Theorem 3.6]{ngai2000approximate}. 
When $\varphi$ is a differentiable at $\bar{z}$, $\partial \varphi(\bar{z})$ reduces to the gradient $\nabla \varphi(\bar{z})$ (see, e.g., \cite[Exercise 8.8]{rockafellar2009variational}). If $\varphi$ is a convex function, the Fr\'{e}chet (regular) subdifferential coincides with the subdifferential defined in the usual convex
analysis sense (see, e.g., \cite[Proposition 8.12]{rockafellar2009variational}), i.e.,
\[
\partial \varphi(\bar{z}) = \{ v \in \mathbb{R}^d ~|~ \varphi(z)  \ge \varphi(\bar{z}) + \langle v, z - \bar{z}\rangle  , \forall z \in \mathbb{R}^d\}.
\]
For a convex set $U$ and a point $\bar z$, we denote by
$ \mathcal{N}_U(\bar z):=\partial \delta_U (\bar z)$
			  the normal cone to $U$ at $\bar  z$.
 		If $\varphi$ is a $\rho$-weakly convex function, then it follows  from \cite[Exercise 8.8]{rockafellar2009variational} that 
\begin{equation}\label{wc_subdiff_sum}
	\partial (\varphi(z)+ \frac{\rho}{2}\|z\|^2) = \partial \varphi(z) +  \rho z .
\end{equation}
Based on this formula, we now extend the partial  subdifferentiation rule for convex functions (case $\rho=0$) in \cite[Proposition 1]{ye2021difference} to the whole class of weakly convex functions ($\rho\geq 0$).
		\begin{proposition}[Partial subdifferentiation]\label{partiald}
			Let $\rho\geq 0$ and $\varphi:\mathbb{R}^n\times \mathbb{R}^m\rightarrow [-\infty,+\infty]$ be a $\rho$-weakly convex function and let $(\bar x,\bar y)$ be a point where $\varphi$ is finite. Then
			\begin{equation} \label{partialsubg} \partial \varphi(\bar x,\bar y) \subseteq \partial_x \varphi(\bar x,\bar y) \times \partial_y \varphi(\bar x,\bar y),\end{equation}
			where $\partial_x \varphi(\bar x,\bar y) $ and $ \partial_y \varphi(\bar x,\bar y)$  denote the partial regular subdifferential of $\varphi$ with respect to $x$ and $y$ at $(\bar x,\bar y)$, respectively.
			The inclusion (\ref{partialsubg}) becomes an equality under one of the following conditions.
			\begin{itemize}
				\item[(a)] For every $\xi \in \partial_x \varphi (\bar x,\bar y)$, it holds that $ \varphi (x,y)-\varphi (\bar x, y) \geq \langle \xi, x-\bar x\rangle - \frac{\rho}{2}\|x-\bar{x}\|^2, \  \forall (x,y) \in \mathbb{R}^n\times \mathbb{R}^m .$
				\item[(b)] $\varphi(x,y) =\varphi_1(x) +\varphi_2(y)$.
				\item[(c)] For any $\varepsilon>0$, there is $\delta>0$ such that 
				\begin{eqnarray}
					\mbox{ either } 	&& \partial_x \varphi (\bar x,\bar y) \subseteq \partial_x \varphi (\bar x, y)  +\varepsilon \mathbb{B}_{\mathbb{R}^n} \quad \forall y\in 
					\mathbb{B}_\delta (\bar y) \label{(7)} \\
					\mbox{ or } 	&& \partial_y \varphi (\bar x,\bar y) \subseteq \partial_y \varphi ( x, \bar y)  +\varepsilon \mathbb{B}_{\mathbb{R}^m} \quad \forall x\in 
					\mathbb{B}_\delta (\bar x),\label{(8)}
				\end{eqnarray}
				where $\mathbb{B}_\delta(\bar x)$ denotes the open ball centered at $\bar x$ with radius equal to $\delta$ and $\mathbb{B}_{\mathbb{R}^n}$ denotes the open unit ball centered at the origin in $\mathbb{R}^n$.
				\item[(d)] $\varphi(x,y)$  is continuously differentiable  respect to one of the variables $x$ or   $y $  at    $(\bar x,\bar y)$.
			\end{itemize}
			Moreover 
			$(b)\Longrightarrow (a), (d) \Longrightarrow (c) \Longrightarrow(a).$
		\end{proposition}

We will now show how to calculate elements in the regular subdifferential of $v_{\gamma}(x,y)$. 
	
			\begin{theorem}\label{Thm4.1} 
Suppose  Assumptions \ref{a1} and \ref{a3} hold.				Let $\rho_f>0, \gamma \in (0, 1/\rho_f)$, $\rho_v  \ge \frac{\rho_f}{1 - \gamma\rho_f}$, $(\bar{x}, \bar{y}) \in X \times \mathbb{R}^m$ and $\tilde{y}\in {S}_{\gamma}(\bar{x},\bar{y})$. Then
				\begin{eqnarray}\label{inclusionvaluef}
					\lefteqn{\partial  v_\gamma(\bar{x},\bar{y}) \supseteq} \nonumber \\
					&& \Bigg \{ \left(\xi_x, (\bar{y} - \tilde{y})/\gamma \right)  ~ \Big| ~ (\xi_x,0)\in   \partial f(\bar x,\tilde y) + \sum_{i = 1}^{l} \eta_i \partial g_i(\bar{x}, \tilde{y}) + \{0\} \times \mathcal{N}_{ Y}(\tilde{y}) +  \{ (0,\tilde{y} - \bar{y})/\gamma )\}, \nonumber \\
					&& \hspace*{230pt}  
					\eta \in \mathbb{R}^l, \eta \ge 0, \sum_{i = 1}^{l}\eta_i g_i(\bar{x},\tilde{y}) = 0  \Bigg \},
				\end{eqnarray} 	
				and the equality holds in (\ref{inclusionvaluef}) provided that 
				\begin{equation} \mathcal{N}_E(\bar{x}, \tilde{y}) = \Big\{ \sum_{i = 1}^{l} \eta_i \partial g_i(\bar{x}, \tilde{y}) + \{0\} \times \mathcal{N}_{ Y}(\tilde{y}) \mid \eta\ge 0, \sum_{i = 1}^{l}\eta_i g_i(\bar{x},\tilde{y}) = 0 \Big\}, \label{normalconeeq}\end{equation}
				
				where $E:=\{(x,y)\in \mathbb{R}^n \times Y: g(x,y)\leq 0\}$. 
				
				Moreover if  the partial derivative formula holds
				\begin{eqnarray}
					\partial f(\bar x,\tilde y)= \partial_x f(\bar x,\tilde y)\times \partial_y f(\bar x,\tilde y) ,&& 
					\partial g_i(\bar x,\tilde y)= \partial_x g_i(\bar x, \tilde y)\times \partial_y g_i(\bar x, \tilde y), \label{partialdnew}
				\end{eqnarray}
				then 
				\begin{equation}\label{valuefinc}
					\bigcup_{ \eta \in \mathcal{M}_\gamma(\bar{x}, \bar{y}, \tilde{y})} \left ( \left\{\partial_x f(\bar{x}, \tilde{y}) + \sum_{i = 1}^{l} \eta_i \partial_x g_i(\bar{x}, \tilde{y}) \right\} \times \left\{(\bar{y} - \tilde{y})/ \gamma \right\} \right )\subseteq \partial v_\gamma(\bar{x},\bar{y}),
				\end{equation}
				where $\mathcal{M}_\gamma(\bar{x}, \bar{y}, \tilde{y})$ denotes the set of Karush-Kuhn-Tucker (KKT) multipliers of the minimization problem in the definition of Moreau envelope $v_{\gamma}$ in \eqref{def_v1}, i.e.,
				\[
				\begin{aligned}
					\mathcal{M}_\gamma(\bar{x}, \bar{y}, \tilde{y})
					:= &\Bigg \{ \eta \in \mathbb{R}^l_+ \Big|  0 \in \partial_y f(\bar{x}, \tilde{y}) + \frac{1}{\gamma}( \tilde{y} - \bar{y}) + \sum_{i = 1}^{l} \eta_i \partial_y g_i(\bar{x}, \tilde{y}) + \mathcal{N}_{Y}(\tilde{y}), \\
					& \hspace{240pt} \sum_{i = 1}^{l}\eta_i g_i(\bar{x}, \tilde{y}) = 0   \Bigg \},
				\end{aligned}
				\]
				and the equality in  (\ref{valuefinc}) holds provided that 
				(\ref{normalconeeq}) holds. 
			\end{theorem}
			\begin{proof} By Theorem  \ref{Thm2},   when 
   $\gamma \in (0, 1/\rho_f)$ and $\rho_v  \ge \frac{\rho_f}{1 - \gamma\rho_f}$,  the function $v_{\gamma}(x,y) + \frac{\rho_v}{2}\|(x,y)\|^2 = \inf_{w \in \mathbb{R}^m} 
    \phi_{\gamma, \rho_v}(x, y, w) $,
    where 
   $$ \phi_{\gamma, \rho_v}(x, y, w)  := \,f(x,w) + \frac{\rho_v}{2}\|(x,y)\|^2 + \frac{1}{2\gamma}\|w - y\|^2 + \delta_C(x,w)$$    is convex and the  extended-valued function $\phi_{\gamma, \rho_v}(x, y, w)$ is convex with respect to $(x,y,w)$ on $\mathbb{R}^n \times \mathbb{R}^m \times \mathbb{R}^m$. 
						Notice that  $$S_{\gamma}(x,y) :=\displaystyle  \mathrm{argmin}_ {w \in Y} \left\{ f(x,w) + \frac{1}{2\gamma}\|w - y\|^2+ \delta_C(x,w)  \right\} = \mathrm{argmin}_{w \in \mathbb{R}^m}   \phi_{\gamma, \rho_v}(x,y,w) .$$		According to \cite[Theorem 24]{rockafellar1974conjugate}, 
      since $\tilde{y} \in S_{\gamma}(x,y) $, we have 
				\begin{equation*}
					\partial \left  (v_\gamma (\cdot)+ \frac{\rho_v}{2}\|\cdot\|^2 \right ) (\bar x,\bar y )= \left \{ (\xi_x, \xi_y)\in \mathbb{R}^n\times \mathbb{R}^m ~|~ (\xi_x,\xi_y,0) \in \partial \phi_{\gamma, \rho_v} (\bar{x}, \bar{y}, \tilde{y}) \right \}.
				\end{equation*}
				Since the function $ \phi_{\gamma, \rho_v}(x,y,w)$ is continuously differentiable with respect to variable $y$
				we have  from Proposition \ref{partiald} that for any $(x,y,w)$,
                \[
                \begin{aligned}
                &\left \{ (\xi_x, \xi_y)\in \mathbb{R}^n\times \mathbb{R}^m ~|~ (\xi_x,\xi_y,0) \in \partial \phi_{\gamma, \rho_v} (\bar{x}, \bar{y}, \tilde{y}) \right \} \\
                = \, & \{ (\xi_x, (\bar{y} - \tilde{y})/\gamma + \rho_v \bar{y})\in \mathbb{R}^n\times \mathbb{R}^m  ~|~  
					 (\xi_x,(\bar{y} - \tilde{y})/\gamma + \rho_v \bar{y},0) \in \partial \phi_{\gamma, \rho_v} (\bar{x}, \bar{y}, \tilde{y}) \}.
                \end{aligned}
                \]  
                It follows that	\begin{eqnarray*}
					\lefteqn{\partial \left  (v_\gamma (\cdot)+ \frac{\rho_v}{2}\|\cdot\|^2 \right ) (\bar x,\bar y )}\\
     &&= \left \{ (\xi_x, (\bar{y} - \tilde{y})/\gamma + \rho_v \bar{y}) ~|~  \xi_x \in \mathbb{R}^n, 
					 (\xi_x,(\bar{y} - \tilde{y})/\gamma + \rho_v \bar{y},0) \in \partial \phi_{\gamma, \rho_v} (\bar{x}, \bar{y}, \tilde{y}) \right \}.
				\end{eqnarray*} 
				By \eqref{wc_subdiff_sum} we have
				$$\partial v_\gamma (\bar x,\bar y)=\partial \left  (v_\gamma (\cdot)+ \frac{\rho_v}{2}\|\cdot\|^2 \right ) (\bar x,\bar y )-\rho_v (\bar x,\bar y).$$
				Therefore \begin{equation}
					\partial v_\gamma(\bar{x},\bar{y}) = \left\{ (\xi_x-\rho_v\bar{x}, (\bar{y} - \tilde{y})/\gamma ) ~|~ \xi_x \in \mathbb{R}^n, \left(\xi_x,(\bar{y} - \tilde{y})/\gamma + \rho_v \bar{y},0\right) \in \partial \phi_{\gamma, \rho_v} (\bar{x}, \bar{y}, \tilde{y}) \right\} \label{gvaluefunction}
				\end{equation} holds.    And because $v_{\gamma}(x,y) + \frac{\rho_v}{2}\|(x,y)\|^2 = \inf_{w \in \mathbb{R}^m}  \left\{ \phi_{\gamma, \rho_v}(x, y, w) \right\}$ and
				\[
				\begin{aligned}
					 &\partial \phi_{\gamma, \rho_v}(x, y, w) = \left\{ (\zeta_x + \rho_v x , \rho_v y + (y - w)/\gamma, \zeta_w + (w - y)/\gamma )~|~ (\zeta_x, \zeta_w) \in \partial \left(f + \delta_C\right)(x,y) \right\},
				\end{aligned}
				\]
				the conclusion can be proved by using the same arguments as in the proof of Theorem 3 in \cite{ye2021difference}.
			\end{proof}

This result can be viewed as an extension of Theorem 3 in \citet{ye2021difference}. When $\rho_f=0$ (i.e. $f$ is convex in both $x$ and $y$) and $\gamma =\infty $, setting $\bar{y}=\tilde{y}$, defining $1/\gamma=0$, $\mathcal{M}_\infty(\bar{x}, \bar{y}, \tilde{y})=\mathcal{M}(\bar{x}, \bar{y})$, where $\mathcal{M}(\bar{x}, \bar{y})$ represents the set of multipliers at $\bar y$ for $(P_{\bar x})$, Theorem 4 reduces to Theorem 3 in \citet{ye2021difference}.

\section{iP-DwCA: Inexact Proximal Difference of Weakly Convex Algorithm}
\label{sec:alg}

\subsection{Algorithm design}
\label{sec:alg-design}
We introduce one last assumption. 

\begin{assumption} $X$ and $Y$ are convex sets, and $F(x, y) = F_1(x, y) - F_2(x, y)$, where $F_1(x, y)$ and $F_2(x, y)$ are $\rho_F$-weakly convex on $X \times Y$. \label{a4}  
\end{assumption} 

We are now ready to develop an algorithm to solve $(BP)$ under Assumptions \ref{a1}--\ref{a4}. Consider the following problem class indexed by $\epsilon \geq 0$: 
\begin{equation*}\label{general_reformulated_problem_epsilon}
		{\rm (VP)}_\gamma^\epsilon~~~~~~~~~~\begin{aligned}
			\min_{(x,y)\in C}  ~~& F(x,y) \\
			s.t. ~~~~& f(x,y)-v_\gamma (x,y)\leq \epsilon.
		\end{aligned}
	\end{equation*} 

 Under Assumption \ref{a1},  ${\rm (VP)}_\gamma^\epsilon$ for $\epsilon = 0$ is equivalent to {\rm (BP)} (Theorem \ref{thm_reformulate}). When $\epsilon > 0$, ${\rm (VP)}_\gamma^\epsilon$ is a relaxation of {\rm (BP)}. It is clear that the lower level solution of ${\rm (VP)}_\gamma^\epsilon$ is $\epsilon$-optimal. The following result further establishes that for any $\delta > 0$, there exists $\epsilon > 0$ such that there is a local minimizer of the relaxed problem ${\rm (VP)}_\gamma^\epsilon$ that is $\epsilon$-close to the solution set of {\rm (BP)}. 			
			\begin{proposition} 
   \label{prop:prox}
				Suppose  $\mathcal{S}^*$,  the solution set of problem {\rm (BP)},  is nonempty and compact, and either $\gamma =\infty$ or $\rho_f>0$ and $\gamma \in (0, 1/\rho_f)$.  Then for any $\delta > 0$, there exists $\bar{\epsilon} > 0$ such that for any $\epsilon \in (0,\bar{\epsilon}]$, there exists $(x_{\epsilon},y_{\epsilon})$ which is  a local minimum of $\epsilon$-approximation problem ${\rm (VP)}_\gamma^\epsilon$ with  $\mathrm{dist}((x_{\epsilon},y_{\epsilon}),\mathcal{S}^*) < \delta$.
				
			\end{proposition}
We omit the proof of Proposition \ref{prop:prox}, as it is a direct extension of Proposition 6 in \citet{ye2021difference}.

To solve ${\rm (VP)}_\gamma^\epsilon$, we rewrite ${\rm (VP)}_\gamma^\epsilon$ as:
 \begin{equation*} 
		{\rm DwC} - {\rm (VP)}_\gamma^\epsilon~~~~~~~~~~\begin{aligned}
			\min_{(x,y)\in C}  ~~& \left (F_1(x,y) + \frac{\rho_F}{2} \|(x,y)\|^2 \right )  - \left (F_2(x,y) + \frac{\rho_F}{2} \|(x,y)\|^2 \right ) \\
			s.t. ~~~~& \left(f(x,y) +\frac{\rho_v}{2}\|(x,y)\|^2 \right) - \left ( v_\gamma (x,y) +\frac{\rho_v}{2}\|(x,y)\|^2 \right) \leq \epsilon. \label{reformulation-DwC}
		\end{aligned}
 \end{equation*} 
In the above, we use Assumption \ref{a4} to rewrite the objective function as the difference of two convex functions.
Furthermore,  choosing $\gamma \in (0, 1/\rho_f)$ makes 
 $$\left( f(x,y)+ \frac{\rho_v}{2}\|(x,y)\|^2 \right) - \left( v_\gamma(x,y)+ \frac{\rho_v}{2}\|(x,y)\|^2 \right)$$
 a difference of two convex functions, because $v_\gamma(x, y)$ is $\rho_v$-weakly convex (Theorem \ref{Thm2}) and $\rho_v =  \frac{\rho_f}{1- \gamma \rho_f} > \rho_f$. That is, ${\rm DwC} - {\rm (VP)}_\gamma^\epsilon$ is a constrained DC program for $\gamma \in (0, 1/\rho_f)$. 

\cite{ye2021difference} proposed an algorithm called inexact proximal DCA (iP-DCA) for solving general constrained DC programs. At each iteration, iP-DCA linearizes the concave parts of the DC program and adds a proximal term to get a strongly convex optimization problem, then approximately solves that optimization problem. We can carry out the linearization for ${\rm DwC} - {\rm (VP)}_\gamma^\epsilon$ via Theorem \ref{Thm4.1} under Assumptions \ref{a1} and \ref{a3}.

In general, iP-DCA is guaranteed to converge for constrained DC programs satisfying standard constraint qualifications; see Section \ref{DCalg}. While these constraint qualifications never hold for ${\rm (VP)}_\gamma^\epsilon$ with $\epsilon = 0$, a standard constraint qualification holds for $\epsilon > 0$ when the lower level problem $(P_x)$ is strictly convex with respect to the lower level variable $y$; see Proposition \ref{prop4.2} in Section \ref{sec:con}.

Thus, we propose solving (BP) under Assumptions \ref{a1}--\ref{a4} by applying iP-DCA from \citet{ye2021difference} to ${\rm DwC} - {\rm (VP)}_\gamma^\epsilon$ with $\gamma \in (0, 1/\rho_f)$ and $\epsilon > 0$. We call this algorithm ``inexact proximal difference of weakly convex algorithm" (iP-DwCA), and outline its steps in the next subsection. 

\subsection{Algorithm steps}
\label{sec:alg-steps}
We are now ready to describe the inexact proximal DwC algorithm (iP-DwCA) to solve 
DwC-${\rm (VP)}_\gamma^\epsilon$. Given a current iterate $(x^k, y^k)$ for each $k=0,1,\ldots$, we first solve the proximal lower-level problem parameterized  by $x^k$ and $y^k$, and with $\gamma \in (0,\frac{1}{\rho_f})$,
\begin{equation}\label{LL-iter}
	\min_{y \in Y}   f(x^k,y) + \frac{1}{2\gamma}\left\|y - y^k \right\|^2, ~ s.t.~ g(x^k,y) \le 0,
\end{equation}
to obtain the value $v_\gamma(x^k,y^k)$, a solution $\tilde{y}^k \in {S}_\gamma(x^k, y^k)$ and a corresponding KKT multiplier $\eta^k \in \mathcal{M}_\gamma(x^k, y^k,\tilde{y}^k)$,
where $\mathcal{M}_\gamma(x^k, y^k,\tilde{y}^k)$ denotes the set of KKT multipliers of the proximal lower-level problem parameterized  by $x^k$ and $y^k$, i.e.,
\[
\begin{aligned}
	\mathcal{M}_\gamma(x^k, y^k,\tilde{y}^k)
	:= &\Bigg \{ \eta \in \mathbb{R}^l_+ \Big|  0 \in \partial_y f(x^k, \tilde{y}^k) + \frac{1}{\gamma}( \tilde{y}^k - y^k) + \sum_{i = 1}^{l} \eta_i \partial_y g_i(x^k, \tilde{y}^k) + \mathcal{N}_{Y}(\tilde{y}^k), \\
	& \hspace{240pt} \sum_{i = 1}^{l}\eta_i g_i(x^k, \tilde{y}^k) = 0   \Bigg \}.
\end{aligned}
\]
 Select 
\begin{eqnarray}
	\xi_0^k\in \partial F_2(x^k,y^k), \quad \xi_1^k \in \partial_x f(x^k,\tilde{y}^k) + \sum_{i = 1}^{l} \eta^k_i \partial_x g_i(x^k,\tilde{y}^k). \label{subgvaluefunction}
\end{eqnarray}
Then, if the partial derivative formula  (\ref{partialdnew}) holds, by sensitivity analysis, $(\xi_1^k, (y^k - \tilde{y}^k)/\gamma )$ is an element of the subdifferential $ \partial v_{\gamma}(x^k, y^k)$ (see (\ref{valuefinc}) in Theorem \ref{Thm4.1}).
Compute $(x^{k+1},y^{k+1})$ as an approximate minimizer of the strongly convex subproblem for problem DwC-${\rm (VP)}_\gamma^\epsilon$ given by				
\begin{equation} \label{DCA2_subproblem}
	\begin{aligned}
		\underset{(x,y)\in C}{\text{min}} ~~   & \,\,  {\varphi}_k(x,y) :=F_1(x,y) + \frac{\rho_F }{2}\left\|(x,y) \right\|^2 - \left\langle \xi_0^k + \rho_F(x^k,y^k), (x,y) \right\rangle   \\
		& +\beta_k \max\left\{f(x,y) + \frac{\rho_v}{2}\|(x,y)\|^2- V_k(x,y) - \epsilon, 0 \right\} + \frac{\alpha}{2} \left\|(x,y) - (x^k,y^k) \right\|^2, 
	\end{aligned}
\end{equation}
where $\alpha >0$, $\rho_v \ge  \frac{\rho_f}{1 - \gamma\rho_f}$, $\beta_k$ represents the adaptive penalty parameter and
\[
V_k(x,y):= v_\gamma(x^k,y^k) + \frac{\rho_v}{2}\left\|(x^k,y^k) \right\|^2 +  \left\langle \xi_1^{k} + \rho_v x^k, x - x^k \right\rangle + \left\langle (y^k - \tilde{y}^k)/\gamma + \rho_v y^k , y- y^k  \right\rangle.
\]
We introduce an inexact condition for choosing $(x^{k+1},y^{k+1})$:
\begin{equation}\label{inexact}
	\mathrm{dist}\left(0, \partial \varphi_k(x^{k+1},y^{k+1}) + \mathcal{N}_C(x^{k+1},y^{k+1})\right) \le \frac{\sqrt{2}}{2} \alpha \left\|(x^{k},y^{k}) - (x^{k-1},y^{k-1}) \right\|,
\end{equation}
where $\mathrm{dist}(z, \Omega) $ is the distance from $z$ to $\Omega$.
Using above constructions and
letting \begin{equation} 
	t^{k+1} = \max \left\{ f(x^{k+1},y^{k+1}) + \frac{\rho_v}{2}\|(x^{k+1},y^{k+1})\|^2- V_k(x^{k+1},y^{k+1})- \epsilon, 0 \right\},\label{deft}
\end{equation} 
we are ready to present  the iP-DwCA in Algorithm \ref{DCA2}.
Note that when $F_1, F_2, f$ are convex in both $x$ and $y$,  i.e. $\rho_F = 0, \rho_f=0$ in Assumptions \ref{a1} and \ref{a4}, then setting $\gamma = \infty, \rho_v = 0$, and defining
$$\mathcal{M}_\infty({x}, {y}, \tilde{y})=\mathcal{M}({x}, {y})$$ 
in iP-DwCA recovers the iP-DCA proposed and studied in \citet{ye2021difference,pmlr-v162-gao22j}. 
To ensure that the iP-DwCA is well defined, we require following standing assumption throughout this paper.
\begin{assumption}\label{A5}
	$\mathcal{M}_\gamma(x^k, y^k,\tilde{y}^k)$ is nonempty for each $k=0,1,\ldots$.
\end{assumption}
 
\begin{algorithm}[h]
	\caption{iP-DwCA}\label{DCA2}
	\begin{algorithmic}[1]
		\State Take an initial point $(x^0, y^0) \in C$; an initial penalty parameter $\beta_0>0$; $\delta_\beta > 0$;  $\alpha > 0$; $c_\beta > 0$;  $\gamma \in (0,\frac{1}{\rho_f})$ if $\rho_f>0$ and $\gamma=\infty$ if $\rho_f=0$; $\rho_v \ge  \frac{\rho_f}{1 - \gamma\rho_f}$; $\rho_F\geq 0$; tolerance $tol>0$.
		\For{$k=0,1,\ldots$}{
			\begin{itemize}
				\item[1.]  Solve the proximal lower-level problem  defined in (\ref{LL-iter}) and
				obtain $\tilde{y}^k \in S_\gamma(x^k,y^k)$ and $\eta^k \in \mathcal{M}_\gamma(x^k, y^k,\tilde{y}^k)$.
			Compute $\xi^k_i$, $i=0,1$ according to (\ref{subgvaluefunction}).				
				\item[2.] Obtain an inexact solution $(x^{k+1},y^{k+1})$ of (\ref{DCA2_subproblem}).
				\item[3.] Stopping test. Stop if $\max\{ \|(x^{k+1},y^{k+1}) - (x^{k},y^{k})\|,t^{k+1} \} < tol$.
				\item[4.] Penalty parameter update.
				Set
				\begin{equation*}
					\beta_{k+1} = \left\{
					\begin{aligned}
						&\beta_k + \delta_\beta, \qquad &&\text{if}~\max\{\beta_k, 1/t^{k+1}\} < c_\beta\|(x^{k+1},y^{k+1}) - (x^{k},y^{k})\|^{-1}, \\
						&\beta_k, \qquad &&\text{otherwise}.
					\end{aligned}\right.
				\end{equation*}
				\item[5.] Set $k:=k+1$.
		\end{itemize}}
		\EndFor
	\end{algorithmic}
\end{algorithm}

\section{Theoretical Investigations}
\label{sec:theory}

In this section, we conduct convergence analysis of the proposed iP-DwCA for solving 
DwC-$(VP)_\gamma^\epsilon$  towards ``good” quality solutions, i.e.
KKT points of the approximation problem DwC-${\rm (VP)}_\gamma^\epsilon$. Recall from Section \ref{sec:alg} that iP-DwCA applies iP-DCA to the DC program in ${\rm DwC} - {\rm (VP)}_\gamma^\epsilon$ with $\gamma \in (0, 1/\rho_f)$ and $\epsilon > 0$.
Thus, we first consider convergence analysis of iP-DCA for the general class of DC programs, then apply these results to yield convergence analysis of iP-DwCA.

\subsection{Inexact proximal DC algorithms for standard DC program}
\label{DCalg}

In this part, we recall the iP-DCA proposed in \cite{ye2021difference} for solving the following standard DC program:
\begin{eqnarray*}
	({\rm DC})~~~~~~~\min_{ z\in \Sigma} && f_0(z):=g_0(z)-h_0(z)\\
	s.t. && f_1(z):=g_1(z)-h_1(z)\leq 0,
\end{eqnarray*}
where $\Sigma$ is a closed convex subset of $\mathbb{R}^d$ and $g_0(z),h_0(z),g_1(z),h_1(z):\Sigma \rightarrow \mathbb{R}$ are  convex functions.

Given a current iterate $z^k\in \Sigma$ with  $k=0,1,\ldots$, we select  $\zeta^k_i\in \partial h_i(z^k)$, for $i=0,1$. 
Then we solve the following subproblem approximately and select $z^{k+1}$ as an approximate minimizer:
\begin{equation} \label{subp} 
	\begin{aligned}
		\min_{z \in \Sigma} ~ {\varphi}_k(z) :=\; & g_0(z) - h_0(z^k) -\langle \zeta_0^k, {z-z^k} \rangle \\
		&+\beta_k \max\left\{g_1(z) -h_1(z^k)- \langle\zeta_1^k, z-z^k\rangle, 0\right\} +\frac{\alpha}{2} \left\|z-z^k \right\|^2,
	\end{aligned}
\end{equation}
where $\alpha$ is a given positive constant and $\beta_k $ represents the adaptive penalty parameter.

Choose $z^{k+1}$ as an {\em approximate minimizer} of the convex subproblem (\ref{subp}) satisfying the following criterion
\begin{equation}\label{inexact2}
	\mathrm{dist}\left(0, \partial  {\varphi}_k( z^{k+1}) + \mathcal{N}_\Sigma(z^{k+1}) \right) \le \frac{\sqrt{2}}{2} \alpha \left\|z^k - z^{k-1} \right\|.
\end{equation}
Using  above constructions and	letting $t^{k+1} := \max\{g_1(z^{k+1}) -h_1(z^k)- \langle\zeta_1^k, z^{k+1} - z^k\rangle, 0\}$, we are ready to present the iP-DCA in Algorithm \ref{ipDCA}. 
	\begin{algorithm}[h]
		\caption{iP-DCA}\label{ipDCA}
		\begin{algorithmic}[1]
			\State Take an initial point $z^0\in \Sigma$; an initial penalty parameter $\beta_0>0$; $\delta_\beta > 0$;  $\alpha > 0$; $c_\beta > 0$; tolerance $tol>0$.
			\For{$k=0,1,\ldots$}{
				\begin{itemize}
					\item[1.] Compute $\zeta^k_i\in \partial h_i(z^k)$, $i=0,1$.
					\item[2.] Obtain an inexact solution $z^{k+1}$ of (\ref{subp}) satisfying  \eqref{inexact2}.
					\item[3.] Stopping test. Stop if $\max \{ \|z^{k+1}-z^k\|, t^{k+1}\} <tol$.
					\item[4.] Penalty parameter update.
					Set
					\begin{equation*}
						\beta_{k+1} = \left\{
						\begin{aligned}
							&\beta_k + \delta_\beta, \qquad &&\text{if}~\max\{\beta_k, 1/t^{k+1}\} < c_\beta\|z^{k+1}-z^k\|^{-1}, \\
							&\beta_k, \qquad &&\text{otherwise}.
						\end{aligned}\right.
					\end{equation*}
					\item[5.] Set $k:=k+1$.
			\end{itemize}}
			\EndFor
		\end{algorithmic}
	\end{algorithm}

The convergence analysis for iP-DCA is similar to that in \citet{ye2021difference,pmlr-v162-gao22j}, where a more specific setting was considered. 
Similar to \citet{ye2021difference}, we start by providing some solution quality characterizations for problem (DC).

	\begin{definition}\label{Defn3.1} Let $\bar z$ be a feasible solution of problem (DC). We say that $\bar z$ is a stationary/KKT point of problem (DC) if there exists a multiplier $\lambda\geq 0$ such that 
		\begin{eqnarray*} 
			&& 0\in \partial g_0(\bar z) -\partial h_0(\bar z) +\lambda( \partial g_1(\bar z) -\partial h_1(\bar z))+\mathcal{N}_\Sigma(\bar z),\\
			&& (g_1(\bar z)-h_1(\bar z))\lambda=0.
		\end{eqnarray*}
	\end{definition}
	
	\begin{definition}\label{Defn3.2} Let $\bar z$ be a feasible point of problem (DC). We say that the nonzero abnormal multiplier constraint qualification (NNAMCQ) holds at $\bar z$ for problem (DC) if either $f_1(\bar z) <0$ or $f_1(\bar z)= 0$ but 
		\begin{equation}\label{NNAMCQ}
			0 \not \in  \partial g_1(\bar{z}) -\partial h_1(\bar{z}) + \mathcal{N}_\Sigma(\bar{z}).
		\end{equation}
		Let $\bar z\in \Sigma$, we say that the extended no nonzero abnormal multiplier constraint qualification (ENNAMCQ) holds at $\bar z$ for problem (DC) if either $f_1(\bar z) <0$ or $f_1(\bar z)\geq 0$ but
		(\ref{NNAMCQ}) holds.
	\end{definition}
	
	Note that NNAMCQ (ENNAMCQ) is equivalent to  MFCQ  (EMFCQ) respectively; see e.g., \cite{Jourani}.
 We recall following optimality condition for problem (DC) from \cite{ye2021difference}.
	\begin{proposition}\label{Thm3.1} Let $\bar z$ be a local solution of problem (DC). If NNAMCQ holds at $\bar z$ and all functions $g_0,g_1,h_0,h_1$ are Lipschitz around point $\bar z$, then $\bar z$ is a KKT point of problem (DC).
	\end{proposition}

    The subsequential convergence of iP-DCA is established in \cite[Theorem 1]{ ye2021difference}.
		We will proceed to show that the subsequential convergence can be further enhanced to sequential convergence under the Kurdyka-\L{}ojasiewicz(KL) property \citep{attouch2009convergence,attouch2010proximal,attouch2013convergence,bolte2014proximal}. 
  
  Let $\eta \in [0, +\infty]$ and $\Phi_\eta$ denote the class of all concave and continuous functions $\phi: [0,\eta) \rightarrow [0, +\infty)$ satisfying the conditions: (a) $\phi(0) = 0$, (b) $\phi$ is $C^1$ on $(0,\eta)$ and continuous at $0$, (c) $\phi'(s) > 0$ for all $s \in (0,\eta)$.
		\begin{definition}[Kurdyka-\L{}ojasiewicz property]\label{KL}
			Let $\sigma : \mathbb{R}^d \rightarrow (-\infty, +\infty]$	be proper and lower semicontinuous.
			The function $\sigma$ is said to have the Kurdyka-\L{}ojasiewicz (KL) property
			at $\bar{z} \in \mathrm{dom}\, \partial \sigma := \{ z \in \mathbb{R}^d\ ~|~  \partial \sigma(z) \neq \varnothing \}	$,	
			if there exist $\eta \in (0,+\infty]$, a neighborhood $\mathcal{Z}$ of $\bar{z}$ and a function $\phi \in \Phi_\eta$, such that for all
			$z \in \mathcal{Z} \cap \{ z \in \mathbb{R}^d\ ~|~  \sigma(\bar{z}) < \sigma(z) <\sigma(\bar{z}) +\eta \}$,
			the following inequality holds\vspace{-2pt}
			\[
			\phi'(\sigma(z) - \sigma(\bar{z})) \mathrm{dist}(0, \partial \sigma(z)) \ge 1.\vspace{-2pt}
			\]
			If $\sigma$ satisfy the KL property at each point of $\mathrm{dom}\, \partial \sigma$ then $\sigma$ is called a KL function.
			
		\end{definition}
		
		In addition, given the KL property around any points in a compact set, the uniformized KL property holds; see Lemma 6 in \citet{bolte2014proximal}.
		\begin{lemma}[Uniformized KL property] \label{uniformKL}
			Given a compact set $D$ and a proper and lower semicontinuous function $\sigma : \mathbb{R}^d \rightarrow (-\infty, +\infty]$, suppose
			that $\sigma$ is constant on $D$ and satisfies the KL property at each point of $D$. Then,
			there exist $\epsilon, \eta$ and $\phi \in \Phi_\eta$ such that for all $\bar{z} \in D$ and	$z \in\{ z \in \mathbb{R}^d\ ~|~  \mathrm{dist}(z, D) < \epsilon, ~~ \sigma(\bar{z}) < \sigma(z) <\sigma(\bar{z}) +\eta \}$, it holds \vspace{-2pt}
			\[
			\phi'(\sigma(z) - \sigma(\bar{z})) \mathrm{dist}(0, \partial \sigma(z)) \ge 1.
			\]
		\end{lemma}
		
		The idea for conducting the sequential convergence analysis is identical to that in \cite{pmlr-v162-gao22j}. Initially, we define the merit function for (DC), which is inspired by \citet{liu2019refined}. The merit function is defined as:
		\begin{equation}\label{meritf}
		{E}_{\beta}(z,z_0, \xi):= {g}_0(z) - \langle \xi_0,z \rangle + {h}_0^*(\xi_0) + \delta_{\Sigma}(z) + \beta \max  \{{g}_1(z) - \langle \xi_1, z \rangle + {h}_1^*(\xi_1), 0\} + \frac{\alpha}{4} \|z-z_0\|^2,
		\end{equation}
		where $h^*$ denotes the conjugate function of $h$, that is, $h^*(\xi) := \sup_z \{\langle \xi,z\rangle - h(z) \}$.

		To analyze sequential convergence, we first establish the sufficient decrease property and the relative error condition of the merit function $E_{\beta}(z,z_0, \xi)$ defined as in (\ref{meritf}), which are summarized in the following lemma. Although the proof of the following lemma closely resembles \citet[Lemma 3.4]{pmlr-v162-gao22j}, we present it in Appendix \ref{appendix_A} for the sake of completeness.
		\begin{lemma}\label{suff_decreasenew}
			Let $\{z^k\}$ be iterates generated by iP-DCA, then	$z^k$ satisfies 
			\begin{equation} 
				E_{\beta_k}(z^{k+1},z^{k},\xi^{k}) + \frac{\alpha}{4} \|z^{k+1} - z^k\|^2 \le E_{\beta_k}(z^k,z^{k-1},\xi^{k-1}),  \label{suff_decrease_eq2} \quad \text{and}
			\end{equation}
			\begin{equation}\label{re_err_eq}
				\mathrm{dist} \left( 0, \partial E_{\beta_k}(z^{k+1},z^{k},\xi^{k}) \right)  \le \,  \frac{\sqrt{2}}{2} \alpha\|z^k - z^{k-1}\| +  (\beta_k+\alpha+1) \|z^{k+1} - z^k\|.
			\end{equation}
		\end{lemma}
  With the the sufficient decrease property and the relative error condition of the merit function $E_{\beta}(z,z_0, \xi)$ established above, we can demonstrate the sequential convergence of iP-DCA under the KL property. 
  
  We now summarize the convergence property in the following theorem. The subsequential convergence result is due to \cite[Theorem 1]{ ye2021difference} and  \cite[Proposition 3]{ ye2021difference} demonstrated that  ENNAMCQ is a sufficient condition for the boundedness of the adaptive penalty parameters sequence $\{\beta_k\}$.  The
 sequential convergence is new but the proof closely follows \citet[Theorem 3.3]{pmlr-v162-gao22j}. For the sake of completeness, we present it in Appendix \ref{appendix_B}.
		
		\begin{theorem}\label{con_KL_alg1}
			Suppose $f_0$ is bounded below on $\Sigma$ and functions $g_0$, $g_1$, $h_0$, $h_1$ are locally Lipschitz on  $\Sigma$. If the  sequences $\{z^k\}$ and $\{\beta_k\}$ generated by   iP-DCA are bounded, then every accumulatoin point of $\{z^k\}$ is a KKT point of problem (DC). Moreover if ENNAMCQ holds at any accumulation points of $\{z^k\}$, then the sequence
$\{\beta_k\}$ must be bounded. Furthermore suppose  and the merit function $E_\beta$ is a KL function. Then the sequence $\{z^k\}$ converges to a KKT point of problem (DC). Moreover the boundedness of the sequence  $\{\beta_k\}$ can be ensured if ENNAMCQ holds at the limiting point.
		\end{theorem}

\subsection{Convergence analysis of iP-DwCA}\label{sec:con}

It is clear that DwC-$(VP)_\gamma^\epsilon$ is a special case of problem (DC) with	 \begin{equation}\label{setting}
\begin{aligned}
&z=(x,y), \ \Sigma =C,\\
	&{g}_0(z) :=F_1(z) + \frac{\rho_F}{2}\|z\|^2, \\
	&{h}_0(z):=F_2(z) + \frac{\rho_F}{2}\|z\|^2, \\
	& {g}_1(z):= f(z)+ \frac{\rho_v}{2}\|z\|^2, \\
	& {h}_1(z):= v_\gamma(z) + \frac{\rho_v}{2}\|z\|^2-\epsilon.
\end{aligned}
\end{equation}

Furthermore, we can confirm by comparing the steps of Algorithm \ref{DCA2} to Algorithm \ref{ipDCA} that iP-DwCA (Algorithm \ref{DCA2}) applies iP-DCA (Algorithm \ref{ipDCA}) to solve  DwC-${\rm (VP)}_\gamma^\epsilon$. 

Combining Definition \ref{Defn3.1} and (\ref{wc_subdiff_sum}), we now define the concept of a stationary point for problem DwC-${\rm (VP)}_\gamma^\epsilon$ which can be seen as a candidate for the optimal solutions of the problem.
\begin{definition}
	We say a  point $(\bar{x},\bar{y})$ is a stationary/KKT point of  {\rm DwC}-${\rm (VP)}_\gamma^\epsilon$ with $\epsilon \geq  0$ if there exists $\lambda\geq 0$ such that
	\begin{equation*}\label{ClarkKKT}
		\left\{ \quad \begin{aligned}
			&0 \in \partial F_1(\bar{x},\bar{y}) -\partial F_2(\bar{x},\bar{y})+ \lambda \partial f(\bar{x},\bar{y}) - \lambda \partial v_\gamma(\bar{x}, \bar{y})+ \mathcal{N}_C(\bar{x},\bar{y}), \\
			& f(\bar{x},\bar{y}) - v_\gamma(\bar{x}, \bar{y}) - \epsilon \le 0, \quad  {\lambda} \left(f(\bar{x},\bar{y}) - v_\gamma(\bar{x}, \bar{y}) - \epsilon  \right) = 0.
		\end{aligned}\right.
	\end{equation*}
\end{definition}
Combining Definition \ref{Defn3.2} and (\ref{wc_subdiff_sum}), we define the following constraint qualifications.
			\begin{definition}
				Let $(\bar{x},\bar{y})$ be a feasible solution to ${\rm (VP)}_\gamma$.
				 We say that the nonzero abnormal multiplier constraint qualification (NNAMCQ) holds at $(\bar{x},\bar{y})$  if
				\begin{equation}\label{newNNAMCQ}
					0 \notin \partial f(\bar{x},\bar{y}) - \partial v_{\gamma}(\bar{x},\bar{y}) + \mathcal{N}_C(\bar{x},\bar{y}).
				\end{equation}
				Let $(\bar{x},\bar{y}) \in C$, we say that the extended nonzero abnormal multiplier constraint qualification (ENNAMCQ) holds at $(\bar{x},\bar{y})$ for problem ${\rm (VP)}_\gamma^\epsilon$ if either $f(\bar x,\bar y) - v_{\gamma}(\bar{x},\bar{y}) <\epsilon$ or $f(\bar x,\bar y) - v_{\gamma}(\bar{x},\bar{y}) \ge \epsilon$ but (\ref{newNNAMCQ}) holds.
			\end{definition}
						
			By virtue of Theorems \ref{Thm2} and \ref{Thm3} and Proposition \ref{Thm3.1}, we have the following necessary optimality condition. 
			Since the issue of constraint qualifications for problem ${\rm (VP)}_\gamma$ is complicated and it is not the main concern in this paper, we refer the reader to discussions on this topic in \cite{KuangYe,Yebook}.

			\begin{theorem}\label{convergence1} Assume  Assumptions \ref{a1}--\ref{a4} hold and $\gamma \in (0, 1/\rho_f)$.
				Let $(\bar{x},\bar{y})$ be a local optimal solution to  problem {\rm DwC}-${\rm (VP)}_\gamma^\epsilon$ with $\epsilon \geq  0$. 
				Suppose 
				that $F_1$, $F_2$ and $f$ are all  Lipschitz continuous at $(\bar{x},\bar{y})$, 
				either $\epsilon >  0$ and  NNAMCQ holds or $\epsilon = 0$ and certain constraint qualification holds. Then $(\bar{x},\bar{y})$ is a KKT point of problem  {\rm DwC}-${\rm (VP)}_\gamma^\epsilon$.
			\end{theorem}

By Theorem \ref{con_KL_alg1}, ENNAMCQ is needed to ensure the boundedness of the sequence $\{\beta_k\}$ and hence the convergence of the algorithm iP-DCA applied to the DC program DwC-$(VP)_\gamma$. Unfortunately, similarly to the value function reformulation ${\rm (VP)}$,  NNAMCQ never holds for  ${\rm (VP)}_\gamma$; see e.g., \cite[Proposition 7]{ye2021difference} and its proof. 
Thus, in Section \ref{sec:alg-design}--\ref{sec:alg-steps}, we solve the relaxed DC program DwC-${\rm (VP)}_\gamma^\epsilon$ for a small positive $\epsilon$ instead.

For  $\epsilon>0$, ENNAMCQ is a standard constraint qualification for ${\rm (VP)}_\gamma^\epsilon$; see  Proposition 4.1 in \citet{LinXuYe}. By Proposition 8 in \citet{ye2021difference}, ENNAMCQ holds automatically for problem ${\rm (VP)}_\gamma^\epsilon$ when $\gamma =\infty , \epsilon>0$.  For the case of $\gamma < \infty$, we show that  ENNAMCQ  holds 
for ${\rm (VP)}_\gamma^\epsilon$ with $\epsilon>0$ when the lower-level problem is strictly convex with respect to the lower-level variable $y$.

			\begin{proposition}\label{prop4.2}  Suppose that
  Assumption \ref{a3} hold and
 $\gamma \in (0, 1/\rho_f)$.
				Then for any $(\bar{x},\bar{y}) \in C$, 
				problem $(VP)_\gamma^\epsilon$ 
				 with $\epsilon>0$  satisfies ENNAMCQ at $(\bar{x},\bar{y})$ provided that for $\bar x$, the lower level problem is strictly convex, i.e., $f(\bar{x},y) + \delta_{C}(\bar{x},y)$ is strictly convex with respect to $y$.
			\end{proposition}
			\begin{proof} 
				If $f(\bar{x},\bar{y}) - v_\gamma(\bar{x},\bar{y}) < \epsilon$ holds, then by definition, ENNAMCQ holds at $(\bar{x},\bar{y})$. Now
				suppose that $f(\bar{x},\bar{y}) - v_\gamma(\bar{x},\bar{y})  \ge \epsilon$ and ENNAMCQ does not hold, i.e., 
				\[
				0 \in \partial f(\bar{x},\bar{y}) -  \partial v_\gamma(\bar{x},\bar{y}) + \mathcal{N}_C(\bar{x},\bar{y}).
				\]
				It follows from {the partial subdifferentiation formula }(\ref{partialsubg})   that
				\begin{equation}\label{NNAMCQ_proof_eq1}
					0 \in \begin{bmatrix}
						\partial_x f(\bar{x},\bar{y})  - \partial_x v_\gamma(\bar{x},\bar{y})\\
						\partial_y f(\bar{x},\bar{y}) - \partial_y v_\gamma(\bar{x},\bar{y})
					\end{bmatrix} + \mathcal{N}_C(\bar{x},\bar{y}).
				\end{equation}
				By
				(\ref{partialsubg})  we have
				$$\mathcal{N}_C(\bar{x},\bar{y})=  \partial \delta_C(\bar x,\bar y) \subseteq \partial_x \delta_C(\bar x,\bar y)\times \partial_y \delta_C(\bar x,\bar y)\subseteq  \mathbb{R}^n \times \mathcal{N}_{{\cal F}(\bar{x})}(\bar{y}).$$
				Thus, it follows from \eqref{NNAMCQ_proof_eq1} that 
				\[
				0 \in \partial_y f(\bar{x},\bar{y})  - \partial_y v_\gamma(\bar{x},\bar{y}) + \mathcal{N}_{{\cal F}(\bar{x})}(\bar{y}).
				\]
				Thus 
				Let $\tilde{y}$ be the unique point in the set $S_{\gamma}(\bar{x}, \bar{y})$, it follows from Theorem \ref{Thm4.1} that $\partial_y v_\gamma(\bar{x},\bar{y}) = (\bar{y} - \tilde{y})/\gamma$ and thus 
				\[
				0 \in \partial_y f(\bar{x},\bar{y})  - \left(\bar{y} - \tilde{y}\right)/\gamma + \mathcal{N}_{{\cal F}(\bar{x})}(\bar{y}),
				\]
				and
				\begin{equation}\label{lem2_eq2}
					\left(\bar{y} - \tilde{y}\right)/\gamma \in \partial_y f(\bar{x},\bar{y}) + \mathcal{N}_{{\cal F}(\bar{x})}(\bar{y}).
				\end{equation}
				Since $ \tilde{y} \in S_{\gamma}(\bar{x}, \bar{y}) = \mathrm{argmin}_{w \in Y}\left\{ f(\bar x,w) + \frac{1}{2\gamma}\|w - \bar{y}\|^2 + \delta_{{\cal F}(\bar{x})}(w) \right\}, $ we have 
				\[
				0 \in \partial_y f(\bar{x},\tilde{y})  + \left(\tilde{y} - \bar{y}\right)/\gamma + \mathcal{N}_{{\cal F}(\bar{x})}(\tilde{y}),
				\]
				and
				\begin{equation}\label{lem2_eq3}
					\left(\bar{y} - \tilde{y}\right)/\gamma \in \partial_y f(\bar{x},\tilde{y}) + \mathcal{N}_{{\cal F}(\bar{x})}(\tilde{y}).
				\end{equation}
				As $f(\bar{x},y) + \delta_{C}(\bar{x},y)$ is assumed to be strictly convex with respect to variable $y$, it follows that $\partial_y f(\bar{x},\cdot) + \mathcal{N}_{{\cal F}(\bar{x})}(\cdot)$ is strictly monotone, see, e.g., \cite[Example 22.3]{Heinz-MonotoneOperator-2011}.
				If $\left(\bar{y} - \tilde{y}\right)/\gamma \neq 0$, the strict monotonicity of  $\partial_y f(\bar{x},\cdot) + \mathcal{N}_{{\cal F}(\bar{x})}(\cdot)$, combined with \eqref{lem2_eq2} and \eqref{lem2_eq3}, leads to a contradiction, as shown below
				\[
				 0 = \langle 	\left(\bar{y} - \tilde{y}\right)/\gamma - 	\left(\bar{y} - \tilde{y}\right)/\gamma, \bar{y} - \tilde{y} \rangle > 0.
				\]
				 Consequently, we conclude that $\bar{y} = \tilde{y}\in S_{\gamma}(\bar{x}, \bar{y})$, which implies $f(\bar{x},\bar{y}) = v_\gamma(\bar{x},\bar{y})$.
				However, an obvious contradiction to the assumption that $f(\bar{x},\bar{y}) - v_\gamma(\bar{x},\bar{y}) \ge \epsilon$ occurs and thus the desired conclusion follows immediately. 
			\end{proof}

	By Theorem 	\ref{con_KL_alg1} and Proposition \ref{prop4.2}, 
we have following convergence result for the iP-DwCA in Algorithm \ref{DCA2}.
\begin{theorem}\label{thm:con}  Assume  Assumptions \ref{a1}--\ref{A5} hold.	Suppose $F$ is bounded below on $C$ and functions $F_1$, $F_2$ and $f$ are locally Lipschitz on set $C$. Moreover suppose that  sequences $\{(x^k,y^k)\}$ and $\{\beta_k\}$ generated by the iP-DwCA are bounded. Then every accumulation point of $\{(x^k,y^k)\}$ is a KKT point of problem  {\rm DwC}-${\rm (VP)}_\gamma^\epsilon$. Furthermore, if the merit function ${E}_\beta$ with problem data (\ref{setting}) is a Kurdyka-\L{}ojasiewicz (KL) function (see Definition \ref{KL} for the detailed definition), then the sequence $\{(x^k,y^k)\}$ converges. Moreover either $\gamma=\infty$ or $\gamma \not =\infty$ but  $f(\bar{x},y) + \delta_{C}(\bar{x},y)$ is strictly convex with respect to $y$ for any accumulation point $\bar{x}$ of the  sequence  $\{x^k\}$, then the sequence $\{\beta_k\}$ must be bounded when $\epsilon > 0$.
\end{theorem}

To further justify the validity of our sequential convergence theory in real-world applications, we next discuss the KL property imposed on the merit function ${E}_\beta$ in Theorem \ref{thm:con}.
A large class of functions automatically satisfies the KL property, see  e.g.  \citet{bolte2010characterizations}, \citet{attouch2010proximal}, \citet{attouch2013convergence}, and \citet{bolte2014proximal}. As shown in \citet{bolte2007lojasiewicz} and \citet{bolte2007clarke}, the class of semi-algebraic functions is an important class of functions that meet the KL property.  
Specifically, the class of semi-algebraic functions is closed under various operations, see, e.g., \cite{attouch2010proximal,attouch2013convergence,bolte2014proximal}. In particular, the indicator functions of semi-algebraic sets, finite sum and product of semi-algebraic functions, composition of semi-algebraic functions and partial minimization of semi-algebraic function over semi-algebraic set	are all semi-algebraic functions. Applying this, when the lower-level problem functions $f$ and $g$ are both semi-algebraic, Moreau envelope function $v_\gamma$ is a semi-algebraic function. Moreover, if $F_1$ and $F_2$ are semi-algebraic functions, the merit function ${E}_\beta$ with problem data (\ref{setting})  is also semi-algebraic and hence a KL function. Hence by
Theorem \ref{thm:con}, we have the following result. 
\begin{theorem}\label{semi_con}
	Assume that $F_1$, $F_2$, $f$ and $g$ are semi-algebraic functions. Suppose  Assumptions \ref{a1}--\ref{A5} hold.	Suppose $F$ is bounded below on $C$ and functions $F_1$, $F_2$ and $f$ are locally Lipschitz on set $C$. Moreover suppose the sequences $\{(x^k,y^k)\}$ and $\{\beta_k\}$ generated by  Algorithm \ref{DCA2} are bounded. Then the sequence $\{(x^k,y^k)\}$ converges to a KKT point of problem {\rm DwC}-${\rm (VP)}_\gamma$. Furthermore, if  either $\gamma=\infty$ or $\gamma \not =\infty$ but  $f(\bar{x},y) + \delta_{C}(\bar{x},y)$ is strictly convex with respect to $y$ for any accumulation point $\bar{x}$ of the  sequence  $\{x^k\}$, then the sequence $\{\beta_k\}$ must be bounded when $\epsilon > 0$.\end{theorem}

A broad spectrum of functions encountered in various applications are semi-algebraic, encompassing real polynomial functions and a substantial class of matrix functions. For comprehensive coverage and further details, please refer to Section 4 in \citet{attouch2010proximal}, \citet{attouch2013convergence}, and \citet{bolte2014proximal}, as well as the references therein.

\begin{remark}
	The requirement of semi-algebraic functions stated in Theorem \ref{semi_con} can be relaxed to encompass more general functions that are definable in an o-minimal structure. A comprehensive treatment of this topic can be found in \cite{bolte2006nonsmooth, bolte2007lojasiewicz, bolte2007clarke}, where it is demonstrated that definable sets and functions possess similar properties to semialgebraic sets and functions. When $F_1$, $F_2$, $f$ and $g$ are definable in an o-minimal structure, the merit function ${E}_\beta$ with problem data (\ref{setting})  is also definable, rendering it a KL function, see \citep[Theorem 14]{attouch2010proximal}. Further examples and discussions regarding definable functions can be found in \citet{attouch2010proximal,bolte2016majorization}.
\end{remark}

\section{Bilevel Hyperparameter Tuning Examples}
\label{sec:examples}
We now discuss the main assumptions underlying our approach --  i.e. convexity of the lower level optimization problem (Assumption \ref{a1}), weak convexity of $f$ and $g$ (Assumption \ref{a2new}--\ref{a3}), and $F$ being a difference of weakly convex functions (Assumption \ref{a4}) --  in the specific context of the bilevel optimization problems for tuning hyperparameters in three popular prediction algorithms via a single training and validation split. The discussion for tuning hyperparameters via cross-validation is omitted due to similarity.

   \subsection{Elastic net }
   \label{sec:enet-setup}

Given data $\{(\mathbf{a}_i, b_i)\}_{i=1}^n$ with covariates $\mathbf{a}_i \in \mathbb{R}^p$, continuous response $b_i\in \mathbb{R}$ and a partition of $\{1, 2, \ldots, n\}$ into training examples and validation examples, $I_{\text{tr}}$ and $I_{\text{val}}$, the bilevel hyperparameter selection problem for elastic net regularized regression \citep{zou2003regression} is given by: 
		
			\begin{equation}
				\begin{aligned}
					\min_{\bm{\beta} \in \mathbb{R}^p, {\bm \lambda}\in \mathbb{R}_+^2}\ & 
					\frac12 \sum_{i \in I_{\text{val}}} \left ( b_i - \bm{\beta}^\top \mathbf{a}_i \right )^2 \\
					\text{s.t. } 
					&\bm{\beta} \in \mathop{\arg\min}_{\bm \beta' \in \mathbb{R}^p} 
					\bigg\{ 
					\frac12 \sum_{i\in I_{\text{tr}}} \left ( b_i - {\bm{\beta}'}^\top \mathbf{a}_i \right ) ^2 
					+\lambda_1\|{\bm{\beta}'}\|_1 +  \frac{\lambda_2}{2} \|{\bm{\beta}'}\|_2^2
					\bigg\}.
				\end{aligned}
				\label{eq:elastic1}
			\end{equation}

Assumption \ref{a1}  holds for \eqref{eq:elastic1}, since $f(\bm \lambda, \cdot)$ is a convex function for any $\bm \lambda$. Assumption \ref{a4} holds for \eqref{eq:elastic1}, since $F$ is a convex function of $\bm \lambda$ and $\bm \beta$, and convexity automatically implies weak convexity. Let 
\begin{equation} \label{eq:elasticobj}
f(\bm \lambda, \bm \beta) = \bigg\{ 
					\frac12 \sum_{i\in I_{\text{tr}}} \left ( b_i - {\bm{\beta}}^\top \mathbf{a}_i \right )^2 
					+\lambda_1\|{\bm{\beta}}\|_1 +  \frac{\lambda_2}{2} \|{\bm{\beta}}\|_2^2
					\bigg\},
\end{equation}
denote the objective function of the lower level program in \eqref{eq:elastic1}. 

We will need a mild additional assumption to address Assumptions \ref{a2new} and \ref{a3}: we need to assume that there exists $\bar{\bm \beta} \in \mathbb{R}_+$ such that for any $\bm \lambda \in \mathbb{R}_+^2$, we have:
\begin{align} 
 \left \{  \mathop{\arg\min}_{\bm \beta \in \mathbb{R}^p} f(\bm \lambda, \bm \beta) \right \} =  \left \{ \mathop{\arg\min}_{\bm \beta \in \mathbb{R}^p} f(\bm \lambda, \bm \beta) \text{ s.t. } |\beta_j| \leq \bar \beta_j \text{ for all } j = 1, 2, \ldots, p \right  \} 
\end{align} 
Under this assumption, we can solve \eqref{eq:elastic1} by solving the same problem with the domain of $\bm \beta$ replaced with the set $B = \{ \bm \beta \in \mathbb{R}^p: |\beta_j| \leq \bar \beta_j \text{ for all } j = 1, 2, \ldots, p\}$. 

It now remains to check if $f$ in \eqref{eq:elasticobj} is a weakly convex function of $\bm \lambda$ and $\bm \beta$ on the set $B$. 
First, observe that 
			\[
			\lambda_1 \|\bm{\beta}\|_1 +  \frac{\sqrt{p}}{2}\lambda_1^2 + \frac{\sqrt{p}}{2}\|{\bm{\beta}}\|^2_2 = \sum_{i = 1}^p  \frac{1}{2\sqrt{p}}\left (\lambda_1 + \sqrt{p}|\bm{\beta}_i| \right ) ^2,
			\]
is a composition of a linear function with a convex function. Thus, it follows from Proposition 1.54 in \cite{mordukhovich2013easy} that $\sum_{i = 1}^p \frac{1}{2} (\lambda_1 + \sqrt{p}|\bm{\beta}_i|)^2$ is convex with respect to $(\bm{\beta}, \lambda_1)$ on $\mathbb{R}^p \times \mathbb{R}_+$. This means that $\lambda_1 \|\bm{\beta}\|_1$ is $\sqrt{p}$-weakly convex with respect to $(\bm{\beta}, \lambda_1)$ on $\mathbb{R}^p \times \mathbb{R}_+$. Second, since its Hessian matrix is positive semi-definite,  the function $\frac{\lambda_2}{2} \|{\bm{\beta}}\|_2^2 + \frac{\rho}{2}\lambda_2^2 + \frac{\rho}{2} \|\bm{\beta}\|^2_2$ is convex with respect to $(\bm{\beta}, \lambda_2)$ on $B$ when $\rho \ge \|\bar{\bm{\beta}}\|_2$. We have now shown that the lower level objective function in \eqref{eq:elasticobj} is $\rho_f$-weakly convex with $\rho_f \ge \sqrt{p}+\|\bar{\bm{\beta}}\|_2$, making our proposed algorithm iP-DwCA applicable to \eqref{eq:elastic1}. 

   \subsection{Sparse Group Lasso}
   \label{sec:sgl-setup}

   			Suppose we have data $\{(\mathbf{a}_i^{(1)}, \ldots, \mathbf{a}_i^{(M)}, b_i)\}_{i=1}^n$, where $\mathbf{a}_i^{(m)} \in \mathbb{R}^{p_m}$ for $m = 1, 2, \ldots, M$, and $b_i\in \mathbb{R}$, and we have a partition of $\{1, 2, \ldots, n\}$ into training examples and validation examples, $I_{\text{tr}}$ and $I_{\text{val}}$. We can choose hyperparameters for sparse group lasso regression by solving the following problem:
			\begin{equation}
				\begin{aligned}
					\min_{\bm{\beta}^{(m)} \in \mathbb{R}^{p_m}, \bm{\lambda}\in \mathbb{R}_+^{M+1}}\ & \frac12 \sum_{i \in I_{\text{val}}} \left ( b_i - \sum \limits_{m=1}^M (\bm{\beta}^{(m)} )^\top \mathbf{a}_i^{(m)} \right )^2 \\
					\text{s.t. }\ &\bm{\beta} \in  \mathop{\arg\min}_{ \bm{\tilde \beta}^{(m)}  \in \mathbb{R}^{p_m}} \bigg\{ \frac12 \sum_{i\in I_{\text{tr}}} |  b_i - \sum \limits_{m=1}^M (\bm{\tilde \beta}^{(m)} )^\top \mathbf{a}_i^{(m)}  |^2 + \sum_{m=1}^M \lambda_m \|\tilde{\bm{\beta}}^{(m)}\|_2 + \lambda_{M+1}\|\tilde{\bm{\beta}}\|_1 \bigg\},
				\end{aligned}
				\label{eq:sgl1}
			\end{equation}
   where $\bm \beta \equiv (\bm{\beta}^{(1)}, \ldots, \bm{\beta}^{(M)}$.

Assumption \ref{a1}  holds for \eqref{eq:sgl1}, since the lower level objective function $f(\bm \lambda, \cdot)$ is a convex function for any $\bm \lambda \in \mathbb{R}_+^{M+1}$. Assumption \ref{a4} holds for \eqref{eq:sgl1}, since the upper level objective function  $F$ is a convex function of $\bm \lambda$ and $\bm \beta$, and convexity automatically implies weak convexity. It now remains to check whether $f$ is a weakly convex function of $\bm \lambda$ and $\bm \beta^{(1)}$, \ldots, $\beta^{(M)}$. 
First, since
			\[
			\sum_{m=1}^M \lambda_m \|\hat{\bm{\beta}}^{(m)}\|_2  + \sum_{m=1}^M \frac{1}{2}\lambda_m^2 + \frac{1}{2}\|\hat{\bm{\beta}}\|_2 ^2 =\sum_{m=1}^M \frac{1}{2}(\lambda_m + \|\hat{\bm{\beta}}^{(m)}\|_2)^2,
			\]
			is a composition of a linear function with a convex function,
			it is convex with respect to $(\bm{\beta}, \lambda_1, \ldots, \lambda_M)$ on $\mathbb{R}^p \times \mathbb{R}^M_+$. Consequently, $\sum_{m=1}^M \lambda_m \|\hat{\bm{\beta}}^{(m)}\|_2$ is $1$-weakly convex with respect to $(\bm{\beta}, \lambda_1, \ldots, \lambda_M)$ on $\mathbb{R}^p \times \mathbb{R}^M_+$. As demonstrated earlier, $\lambda_{M+1}\|\hat{\bm{\beta}}\|_1$ is $\sqrt{p}$-weakly convex with respect to $(\bm{\beta}, \lambda_{M+1})$ on $\mathbb{R}^p \times \mathbb{R}_+$. Therefore, the lower-level objective function in \eqref{eq:sgl1} is $\rho_f$-weakly convex with $\rho_f \ge 1 + \sqrt{p}$, making our proposed algorithm iP-DwCA applicable to \eqref{eq:sgl1}. 

   \subsection{RBF Kernel Support Vector Machine} 
   \label{sec:svm-setup}

  Let $\bm{\varphi}_{\sigma}(\cdot)$ denote the RBF kernel feature mapping, which is parameterized by $\sigma$. 
  
  Given data $\{(\mathbf{a}_i, b_i)\}_{i=1}^n$ with covariates $\mathbf{a}_i \in \mathbb{R}^p$, binary response $b_i$ and a partition of $\{1, 2, \ldots, n\}$ into training examples and validation examples, $I_{\text{tr}}$ and $I_{\text{val}}$ and a fixed choice of $C$ and $\sigma$, the optimization problem of RBF kernel SVM on the training data can be expressed as:
			\[
			\min_{\mathbf{w}\in \mathbb{R}^p, c \in \mathbb{R}}
			\bigg\{ \frac12\|\mathbf{w}\|^2 + C\sum_{i \in  I_{\text{tr}}}\max( 1 - b_i(\mathbf{w}^\top \bm{\varphi}_{\sigma}(\mathbf{a}_i) - c),0) \bigg\}.    
			\]
			This optimization problem is a unconstrained convex problem and  by reformulating it as a constrained convex problem, one has the strong duality.  The dual problem is equivalent to the following convex problem:
			\begin{align*}
				\min_{\eta\in \mathbb{R}^{|I_{\mathrm{tr}}|} } \ & \frac12 \sum_{i, j \in  I_{\text{tr}}} 
				 \eta_i \eta_j b_i b_j K_{\sigma}(\mathbf{a}_i, \mathbf{a}_j) - \sum_{i \in  I_{\text{tr}}} \eta_i \\
				\text{s.t. } & \sum_{i \in  I_{\text{tr}}} \eta_i b_i = 0, \ 0\le \eta_i \le C, i \in I_{\mathrm{tr}},
			\end{align*}
   where $K_{\sigma}(\mathbf{a}, \mathbf{a}') = \exp( - \sigma \|\mathbf{a} - \mathbf{a}'\|^2)$.
			Suppose that  $\eta$ is a solution of  the dual problem. Then  the unique optimal solutions of the original kernel SVM can be calculated by
			\begin{equation}\label{def_wb}
				\mathbf{w} = \sum_{i \in  I_{\text{tr}}} \eta_i b_i \bm{\varphi}_{\sigma}(\mathbf{a}_i), \quad     c = \sum_{i \in  I_{\text{tr}}} \eta_i b_i K_{\sigma}(\mathbf{a}_i, \mathbf{a}_{j^*}) - b_{j^*}, \quad \text{for any} ~ j^*\in  I_{\text{tr}} ~ \text{satisfying} ~ 0 < \eta_{j^*} < C.
			\end{equation}
			Thus, by using the SVM hinge loss as the measure of validation accuracy, and using \eqref{def_wb}, we can formulate hyperparameter selection for RBF Kernel SVM as the following bilevel program,
		\begin{equation}\label{eq:svm-explicite}
				\begin{aligned}
					\min_{\sigma \in \Sigma, C \in \mathbf{C},  \eta\in \mathbb{R}^{|I_{\mathrm{tr}}|}}\ & \sum_{j \in  I_{\text{val}}}  \max \left (1 - b_j \left (\sum_{i \in  I_{\text{tr}}} \eta_i b_i K_\sigma(\mathbf{a}_i, \mathbf{a}_j)- \sum_{i \in  I_{\text{tr}}} \eta_i b_i K_{\sigma}(\mathbf{a}_i, \mathbf{a}_{j^*}) + b_{j^*}\right ), 0\right ) \\
					\text{s.t. } & 
					\begin{aligned}[t]
						\eta \in
						\begin{aligned}[t]
							\mathop{\arg\min}_{\eta \in \mathbb{R}^{|I_{\mathrm{tr}}|}}\ & \frac12 \sum_{i,j \in  I_{\text{tr}}} 
							 \eta_i \eta_j b_i b_j K_\sigma(\mathbf{a}_i, \mathbf{a}_j) - \sum_{i \in  I_{\text{tr}}}  \eta_i \\
							\text{s.t. } & \sum_{i \in  I_{\text{tr}}} \eta_i b_i = 0, \ 0\le \eta_i \le C,
						\end{aligned} \\
					\end{aligned}
				\end{aligned}
			\end{equation}	
			where  $j^*:=j^*(\eta)\in  I_{\text{tr}}$ satisfying $0<\eta_{j^*}<C$.

Assumption \ref{a1} holds for \eqref{eq:svm-explicite}, as the lower level problem is a convex optimization problem in $\eta$ for any $\sigma \in \Sigma$ and $C \in \mathbf{C}$. Furthermore, since $K_\sigma$ is twice continuously differentiable with respect to $\sigma$, the lower level objective function is twice continuously differentiable. When $\Sigma$ and $\mathbf{C}$ are bounded, this means that the lower level objective function is weakly convex  on the feasible region of variables $\sigma$ and $\eta$, which establishes the key parts of Assumptions \ref{a2new} and \ref{a3} for \eqref{eq:svm-explicite}. 

Assumption \ref{a4} requires the upper level objective function to be weakly convex. The upper level objective function is difficult to work with, as $j^*$ is a complicated function of $\eta$; see \eqref{def_wb}. However, if we ignore the dependence of $j^*$ on $\eta$, then the upper level objective function in \eqref{eq:svm-explicite} is twice continuously differentiable, and thus weakly convex  on the feasible region of variables $\sigma$, $C$ and $\eta$.

			\section{Numerical Experiments}
			\label{sec:exp}
			In this section, we conduct experiments to evaluate the performance of our proposed algorithm iP-DwCA to tune hyperparameters for elastic net, sparse group lasso, and radial basis function kernel support vector machines (RBF-SVM) via a single training and validation split. We compare the computional time, validation error, and test error of iP-DwCA to grid search, random search, and Bayesian optimization method known as the {\it Tree-structured Parzen Estimator approach (TPE)} (\cite{bergstra2013making}). We also apply iP-DCA to solve the reformulated problem \citep{pmlr-v162-gao22j} when this approach is available.

			All numerical experiments are performed on a personal computer with Intel(R) Core(TM) i9-9900K CPU @ 3.60GHz and 16.00 GB memory. We implement all algorithms in Python, and solve all subproblems using ``Pyomo'', which is an open-source, Python-based optimization modeling package.

			\subsection{Elastic net}
   \label{sec:exp-enet}
			
			Details of the elastic net bilevel problem setup and needed assumptions are in Section \ref{sec:enet-setup}. For our experiment, we set $\bar{\bm{\beta}} = 2 \times \bm{1}_p$, the number of training samples (nTr) $= 100$, the number of validation samples (nVal) $= 100$, the number of test samples (nTest) $=300$, and the numbers of features $(p) = 50, 100$.  The covariates $\mathbf{a}_i\in \mathbb{R}^{p}$ were generated i.i.d. from a $p$-dimensional Gaussian distribution with mean $0_p$ and covariance $\mathrm{Cov}(a_{ij}, a_{ik}) = 0.5^{|j - k|}$. The response $\bm{b}$ was generated by
			\[
			b_i = \bm{\beta}^\top \mathbf{a}_i + \sigma \epsilon_i
			\]
			where $\beta_i \in \{0, 1\}$, $\sum_{i = 1}^{p}\beta_i = 15$, $\bm{\epsilon}$ was sampled from the standard Gaussian distribution, and $\sigma$ was chosen such that the signal-to-noise ratio (SNR) is 2, i.e., $\|\mathbf{b} - \bm{\epsilon}\| / (\sigma \|\bm{\epsilon}\|) = 2$.
			
			Grid search method was implemented on a $30\times 30$ uniform-spaced grid on $[0, 100] \times [0, 100]$. Random search was implemented with 100 uniformly random sampling on $[0, 100] \times [0, 100]$. The spaces used in the TPE method for both $\lambda_1$ and $\lambda_2$ were set as the uniform distribution on $[0, 100]$. The initial values for the upper-level variables in iP-DwCA are determined through a preliminary grid search conducted on a $6\times 6$ uniform-spaced grid on $[0, 100] \times [0, 100]$, with the selection of parameters yielding the smallest upper-level objective value. The initial values of upper-level variables of VF-iDCA is $(10, 10)$; we added a box constraint $[0, 100]$ on all the upper-level variables when implementing these two methods.
			The parameters in iP-DwCA were chosen as $\rho_F = 0$, $\rho_f = 2\sqrt{p}$, $\rho_v = 2\rho_f$, $\gamma = 0.5 / \rho_f$, $\beta_0 = 1$, $\delta_\beta = 1$, $c_\beta = 1$, $\alpha = 0.01$, and $\epsilon = 10^{-6}$. The parameters in VF-iDCA were chosen as $\alpha_0 = 1$, $\delta_\alpha = 5$, $c_\alpha = 1$, $\rho = 0.0001$, $\epsilon = 10^{-6}$.
			Both iP-DwCA and VF-iDCA were terminated when the generated iterate satisfied $\max\{t_k / f(\mathbf{x}^k, \mathbf{y}^k), \|(\mathbf{x}^k, \mathbf{y}^k) - (\mathbf{x}^{k - 1}, \mathbf{y}^{k-1})\|\} \le $ tol for the given tol or reach the maximum iteration number $200$. Specifically, we set tol = $10^{-3}$ for iP-DwCA and tol = $10^{-4}$ for VF-iDCA in the example with $p = 50$, and set tol = $10^{-2}$ for iP-DwCA and tol = $10^{-3}$ for VF-iDCA in the example with $p = 100$.
			Since DC type methods always converge fast in the first few iterations, we also investigated the performance of iP-DwCA and VF-iDCA with an early stop strategy (terminated after 10 iterations for iP-DwCA and 30 iterations for VF-iDCA), and denote them by iP-DwCA-E and VF-iDCA-E, respectively.
			
			All the numerical results were averaged over 100 random trials and reported in Table \ref{table:elastic1} and Table \ref{table:elastic2}, where ``Feasibility" denotes a scaled value of the violation of the value function constraint of the iterates generated by iP-DwCA and VF-iDCA, i.e., $(f(\mathbf{x}^k, \mathbf{y}^k) - v_\gamma (\mathbf{x}^k, \mathbf{y}^k))/|I_{\text{tr}}|$ for iP-DwCA and $(f(\mathbf{x}^k, \mathbf{y}^k) - v(\mathbf{x}^k)/|I_{\text{tr}}|$ for VF-iDCA. When feasibility is non-zero, it means that the attained solution is not on the elastic net solution path. As this could be problematic in some applications where the variable selection power of elastic net is of primary interest, we report two types of validation and test error:  ``Val. Err. Infeas." and ``Test. Err. Infeas.", and ``Val. Err." and ``Test. Err.". Here, ``infeasible" validation and test error are calculated using the lower-level variable values of iterates produced by iP-DwCA and VF-iDCA, respectively. For the ``feasible" validation and test error, we take the hyperparameter value of the iterates produced by iP-DwCA and VF-iDCA, solve the lower-level problem with that hyperparameter value to get the corresponding elastic net estimator, then compute the validation and test error of that estimator. An analogous approach is taken for the simulation results in subsequent subsections on sparse group lasso and kernel SVM. The results are in Tables \ref{table:elastic1} and \ref{table:elastic2}.

\begin{table}[ht]
    \centering 
    \caption{Results on the elastic net hyperparameter selection problem (nTr = 100, nVal = 100, nTest = 300, $p$ = 50)}
    \label{table:elastic1}
    \resizebox{\textwidth}{!}{
    \begin{tabular}{c cccccc}
        \toprule 
Method     &      Time(s) & Val. Err. & Test. Err. & Val. Err. Infeas. & Test. Err. Infeas. &  Feasibility\\
\midrule
 iP-DwCA-E &   4.4$\pm$0.5 &   4.6$\pm$0.9 &   4.5$\pm$0.8 &   2.8$\pm$0.6 &   4.0$\pm$0.5 &   0.5$\pm$0.3  \\
 iP-DwCA   &  38.4$\pm$8.3 &   4.6$\pm$0.9 &   4.5$\pm$0.8 &   4.2$\pm$0.8 &   4.4$\pm$0.7 &   0.01$\pm$0.01  \\
 VF-iDCA-E &   4.9$\pm$0.4 &   4.5$\pm$0.9 &   4.5$\pm$0.8 &   4.3$\pm$0.8 &   4.4$\pm$0.7 &   0.01$\pm$0.00  \\
 VF-iDCA   &  33.6$\pm$3.8 &   4.5$\pm$0.9 &   4.5$\pm$0.7 &   4.5$\pm$0.9 &   4.5$\pm$0.7 &   0.00$\pm$0.00  \\
 Grid      &  54.7$\pm$5.5 &   4.5$\pm$0.9 &   4.5$\pm$0.8 & - &  - &  - \\
 Random    &  57.9$\pm$4.4 &   4.5$\pm$0.9 &   4.5$\pm$0.8 & - &  - &  - \\
 TPE       &  58.8$\pm$5.0 &   4.5$\pm$0.9 &   4.5$\pm$0.8 & - &  - &  - \\
\bottomrule 
    \end{tabular}}
\end{table}
   
\begin{table}[ht]
    \centering 
    \caption{Results on the elastic net hyperparameter selection problem  (nTr = 100, nVal = 100, nTest = 300, $p$ = 100)}
    \label{table:elastic2}
    \resizebox{\textwidth}{!}{
    \begin{tabular}{c cccccc}
        \toprule 
Method     &      Time(s) & Val. Err. & Test. Err. & Val. Err. Infeas. & Test. Err. Infeas. &  Feasibility\\
\midrule
 iP-DwCA-E &   8.2$\pm$0.8 &  4.3$\pm$0.8 &  4.4$\pm$0.7 &  1.9$\pm$0.6 &  4.0$\pm$0.6 &   0.8$\pm$0.5  \\
 iP-DwCA   &  64.6$\pm$23.1&  4.3$\pm$0.9 &  4.4$\pm$0.8 &  3.5$\pm$0.7 &  4.0$\pm$0.7 &   0.02$\pm$0.05  \\
 VF-iDCA-E &   8.5$\pm$0.6 &  4.4$\pm$1.1 &  4.6$\pm$0.9 &  3.6$\pm$0.6 &  4.2$\pm$0.7 &   0.03$\pm$0.04  \\
 VF-iDCA   &  38.9$\pm$13.7&  4.3$\pm$1.0 &  4.4$\pm$0.8 &  4.1$\pm$0.9 &  4.4$\pm$0.8 &   0.00$\pm$0.00  \\
 Grid      &  99.7$\pm$8.6 &  4.2$\pm$0.8 &  4.3$\pm$0.7 & - &  - &  - \\
 Random    & 105.9$\pm$7.4 &  4.2$\pm$0.8 &  4.4$\pm$0.7 & - &  - &  - \\
 TPE       & 107.2$\pm$7.5 &  4.2$\pm$0.8 &  4.3$\pm$0.7 & - &  - &  - \\
        \bottomrule 
    \end{tabular}}
\end{table}

In Tables \ref{table:elastic1} and \ref{table:elastic2}, the early-stopping versions of both bilevel programming-based algorithms -- our new method iP-DwCA and the earlier version VF-iDCA, which is a special case of our framework (see Sections \ref{sec:moreau-reform} and \ref{sec:alg}) -- are significantly faster than grid search, random search, and TPE. Without early stopping, the bilevel programming-based algorithms achieve moderate computational time gains. All methods achieve comparable validation and test error. However, interestingly, we see that the bilevel algorithms -- especially with early stopping -- sometimes lead to $\beta$ solutions that are not on the elastic net regularization path, but have lower test error. This could potentially be useful if the practitioner is applying elastic net with an eye to prediction accuracy only, but not if the practitioner is applying elastic net for variable selection, as there is no guarantee that the attained $\beta$ solution has the sparsity properties of elastic net. We conclude by commenting that iP-DwCA always achieved the given stopping criterion, which supports the established convergence results.

			\subsection{Sparse group Lasso}
			
			We tested the sparse group Lasso hyperparameter selection problem (see Section \ref{sec:sgl-setup}) with varying numbers of training samples (nTr), validation samples (nVal), feature dimensions ($p$), and number of groups ($M$). The number of test samples (nTest) was always set to $300$. The datasets employed in the numerical experiments of the sparse group Lasso problem were generated as follows. Each $\mathbf{a}_i \in \mathbb{R}^{p}$ was sampled from the standard normal distribution, and the response $\bm{b}$ was generated by:
			\[
			b_i = \bm{\beta}^\top \mathbf{a}_i + \sigma \epsilon_i
			\]
			where $\bm{\beta}^\top = \big[ \bm{\beta}^{(1)\top}, \dots, \bm{\beta}^{(5)\top} \big]$ with $\bm{\beta}^{(i)\top} = [0+i, 1+i, 2+i, 3+i, 4+i, 0,\dots, 0]\in R^{\frac{p}5}$ for $i = 1,\dots,5$, $\bm{\epsilon}$ was generated from the standard normal distribution, and $\sigma$ was chosen such that the SNR is 2. The groups ${\bm{\beta}}^{(m)}$ in \eqref{eq:sgl1} were selected as ${\bm{\beta}}^{(m)} = (\beta_{(m-1)p/M+1}, \ldots, \beta_{(m-1)p/M+M} )$.
			
			Grid search was implemented with $\lambda_1 = \lambda_2 = \cdots = \lambda_{M} = p_1$ and $\lambda_{M+1} = p_2$ over a $10 \times 10$ uniformly spaced grid on $[5, 500] \times [5, 500]$. Random search was implemented using 100 uniformly random samples on $[5, 500]^M \times [5, 500]$. For the Tree-structured Parzen Estimator (TPE) method, we selected a uniform distribution on $[5, 500]$ for $\lambda_i, i = 1, \dots, M$, and a uniform distribution on $[5, 500]$ for $\lambda_{M+1}$. For iP-DwCA, the initial values for $\lambda_1 ,\dots, \lambda_M$ and $\lambda_{M+1}$ are determined through an initial coarse grid search conducted on a $5\times 5$ uniform-spaced grid on $[5, 500] \times [5, 500]$. We have imposed box constraints of $[5, 500]$ on $\lambda_1, \dots, \lambda_M$, as well as on $\lambda_{M+1}$. For VF-iDCA, we used 1 as the initial value for $r_1, \dots, r_{M+1}$, and added box constraints $[0, 100]$ on $r_1,\dots, r_{M}$ and a box constraint $[0,100]$ on $r_{M+1}$.
			
			The parameters for iP-DwCA were set as $\rho_F = 0$, $\rho_f = 1+ \sqrt{p}$, $\rho_v = 2\rho_f$, $\gamma = 0.5 / \rho_f$, $\beta_0 = 1$, $\delta_\beta = 0.1$, $c_\beta = 1$, $\alpha = 0.001$, and $\epsilon = 10^{-6}$. For VF-iDCA, the parameters were chosen as $\alpha_0 = 1$, $\delta_\alpha = 5$, $c_\alpha = 1$, $\rho = 10^{-4}$, and $\epsilon = 10^{-6}$. The maximum iteration number for iP-DwCA and VF-iDCA was set to 100. Both iP-DwCA and VF-iDCA were terminated when the generated iterate satisfied $\max\{t_k / f(\mathbf{x}^k, \mathbf{y}^k), \|(\mathbf{x}^k, \mathbf{y}^k) - (\mathbf{x}^{k - 1}, \mathbf{y}^{k-1})\|\} \le $ tol. Specifically, we set tol = 1 for iP-DwCA, and tol = 0.05 for VF-iDCA.
			
			We also investigated the performance of iP-DwCA and VF-iDCA with an early stop strategy (terminated after 5 iterations for iP-DwCA and 20 iterations for VF-iDCA) and denoted them as iP-DwCA-E and VF-iDCA-E, respectively. All results were averaged over 100 random trials and reported in Table \ref{table:sgl1} and Table \ref{table:sgl2}, where ``Feasibility" represents the violation of the value function constraint for the iterates generated by iP-DwCA and VF-iDCA, i.e., $(f(\mathbf{x}^k, \mathbf{y}^k) - v_\gamma (\mathbf{x}^k, \mathbf{y}^k))/|I_{\text{tr}}|$ for iP-DwCA and $(f(\mathbf{x}^k, \mathbf{y}^k) - v(\mathbf{x}^k))/|I_{\text{tr}}|$ for VF-iDCA. Moreover, we use "Val. Err. Infeas." and "Test. Err. Infeas." to denote the validation error and test error calculated directly using the lower-level variable values of iterates produced by iP-DwCA and VF-iDCA, respectively; see the discussion in Section \ref{sec:exp-enet}

\begin{table}[ht]
    \centering 
    \caption{Results on the sparse group Lasso hyperparameter selection problem  (nTr = 300, nVal = 300, nTest = 300, $p$ = 90, $M$ = 30) }
    \resizebox{\textwidth}{!}{
    \begin{tabular}{c cccc ccc}
        \toprule 
Method     &      Time(s) & Val. Err. & Test. Err. & Val. Err. Infeas. & Test. Err. Infeas. &  Feasibility\\
\midrule
iP-DwCA-E &  42.3$\pm$55.5& 106.1$\pm$10.2& 109.2$\pm$ 9.4&  82.3$\pm$ 6.7& 100.4$\pm$ 8.1&  12.01$\pm$  2.45\\
iP-DwCA   & 129.3$\pm$111.8& 106.0$\pm$10.1& 109.1$\pm$ 9.4&  97.5$\pm$ 9.2& 104.8$\pm$ 8.7&   0.99$\pm$  0.01\\
VF-iDCA-E &  11.2$\pm$ 0.4&  93.1$\pm$ 8.4& 104.0$\pm$ 9.0&  87.3$\pm$ 7.4& 102.3$\pm$ 8.5&   0.41$\pm$  0.23\\
VF-iDCA   &  34.2$\pm$11.3&  91.5$\pm$ 8.1& 103.1$\pm$ 8.7&  89.6$\pm$ 7.9& 102.5$\pm$ 8.6&   0.05$\pm$  0.01\\
Grid      &  71.0$\pm$ 2.0& 106.1$\pm$10.2& 109.1$\pm$ 9.4& - &  - &  - \\
Random    &  79.0$\pm$ 2.0& 103.1$\pm$ 9.6& 107.5$\pm$ 9.1& - &  - &  - \\
TPE       &  93.1$\pm$ 2.3& 100.6$\pm$ 9.3& 106.2$\pm$ 8.7& - &  - &  - \\
        \bottomrule 
    \end{tabular}}
    \label{table:sgl1}
\end{table}

\begin{table}[ht]
    \centering 
    \caption{Results on the sparse group Lasso hyperparameter selection problem  (nTr = 300, nVal = 300, nTest = 300, $p$ = 150, $M$ = 30)}
    \resizebox{\textwidth}{!}{
    \begin{tabular}{c cccccc}
        \toprule 
Method     &      Time(s) & Val. Err. & Test. Err. & Val. Err. Infeas. & Test. Err. Infeas. &  Feasibility\\
\midrule 
iP-DwCA-E &  22.2$\pm$36.9& 107.5$\pm$10.2& 109.1$\pm$10.1&  79.0$\pm$ 7.7& 100.1$\pm$ 8.4&  15.04$\pm$  2.74\\
iP-DwCA   & 159.3$\pm$107.5& 107.5$\pm$10.2& 109.1$\pm$10.1&  98.1$\pm$ 9.4& 104.6$\pm$ 9.3&   1.00$\pm$  0.05\\
VF-iDCA-E &  25.1$\pm$ 1.0&  93.1$\pm$ 9.1& 104.7$\pm$ 9.0&  85.3$\pm$ 8.0& 102.9$\pm$ 8.7&   0.58$\pm$  0.28\\
VF-iDCA   &  86.4$\pm$26.3&  91.6$\pm$ 8.9& 102.7$\pm$ 8.8&  89.3$\pm$ 8.7& 102.1$\pm$ 8.8&   0.05$\pm$  0.01\\
Grid      & 160.9$\pm$ 9.3& 107.4$\pm$10.2& 108.9$\pm$10.1& - &  - &  - \\
Random    & 193.3$\pm$10.4& 106.0$\pm$10.2& 109.2$\pm$ 9.9& - &  - &  - \\
TPE       & 203.1$\pm$ 9.2& 103.4$\pm$10.1& 107.8$\pm$ 9.8& - &  - &  - \\
        \bottomrule 
    \end{tabular}}
    \label{table:sgl2}
\end{table}

In Tables \ref{table:sgl1} and \ref{table:sgl2}, the early-stopping versions of both bilevel programming-based algorithms are significantly faster than grid search, random search, and TPE. Without early stopping, iP-DwCA is somewhat slower, while iP-DCA remains much faster than grid search, random search, and TPE. With or without early stopping, the bilevel programming-based algorithms achieve relatively close validation and test error to grid search, random search, and TPE. Finally, we again see that bilevel algorithms lead to $\beta$ solutions that are not on the group lasso regularization path, but have comparable or better test error. The same caveats as in Section \ref{sec:exp-enet} apply here: this could potentially be useful if the practitioner is applying group lasso with an eye to prediction accuracy only, but not if the practitioner is applying group lasso for variable selection, as there is no guarantee that the attained $\beta$ solution has the sparsity properties of group lasso. In the setting that prediction accuracy is of primary interest, the early stopping version of our bilevel algorithms are of special interest, as they can converge to a solution much faster than applying grid search, random search, or TPE.

			\subsection{RBF Kernel Support Vector Machine}
		
			We tested the RBF Kernel SVM hyperparameter selection problem (see Section \ref{sec:svm-setup}) using the kernel function $K_{\sigma}(\mathbf{a}, \mathbf{a}') = \exp( - \sigma \|\mathbf{a} - \mathbf{a}'\|^2)$, with $\Sigma = [0.01, 10]$ and $\mathbf{C} = [0.01, 10]$.
			
			Synthetic data used in the experiments were generated as follows: each dataset contained 100 training samples, 100 validation samples, and 300 test samples. As shown in Figure \ref{fig:svm_ref}, we first generated 500 points within a square centered at the origin with a side length of $4\sqrt{\mathrm{arcsec}\sqrt5}$ (the black square), then constructed a region as the union of two 2 by 1 ellipses (the blue curve) and labeled the points inside the region (red points) as -1 and the points outside the region (gray points) as 1. Lastly, we added noise $\epsilon$, generated from $\mathcal{N}(0, 0.2)$, to all data points.
			
			We implemented grid search for ($\sigma$, $C$) over a $10\times 10$ uniformly-spaced grid on $[0.1, 10] \times [1, 10]$. Random search was implemented for ($\sigma$, $C$) with 100 uniformly random samples on $[0, 10] \times [0, 10]$. For the TPE method, both $C$ and $\sigma$ utilized the uniform distribution on $[0, 10]$. 
   
   To implement iP-DwCA (Algorithm \ref{DCA2}), we need to initialize with the modulars for the weakly convex lower and upper level objective functions, $\rho_f$ and $\rho_F$. However, in this case, calculating the modulars analytically is challenging. Thus, we used large values in hopes of overestimating the modulars: we set $\rho_F = \rho_f = 50$ in iP-DwCA. Other parameters in iP-DwCA were chosen as $\rho_v = 2 \rho_f$, $\sigma = .5/\rho_f$, $\beta_0 = 0.0001$, $\delta_\beta = 0.0005$, $c_\beta = 1$, $\epsilon = 0.0001$, and $\alpha = 0.0001$. We deal with the fact that the upper level objective function depends on an index $j^*$ which depends on the lower level solution (as discussed in Section \ref{sec:svm-setup}) with a heuristic approach: in each iteration $k$, after solving the lower level problem to obtain $\eta^k$, we choose an index $j^*$ such that $0 < \eta_{j^*}^k < C^k$.
   
   The initial values in iP-DwCA were selected as $C_0 = 1, \sigma_0 = 0.1$. iP-DwCA was terminated when the generated iterate satisfied $\max\{t_k / f(\mathbf{x}^k, \mathbf{y}^k), \|(\mathbf{x}^k, \mathbf{y}^k) - (\mathbf{x}^{k - 1}, \mathbf{y}^{k-1})\|\} \le $ tol, with tol = 0.01. 
			We also investigated the performance of iP-DwCA with an early stop strategy (terminated after 5 iterations) and denoted it as iP-DwCA-E. 
			
			All results were averaged over 100 random trials and reported in Table \ref{table:svm}, where ``Feasibility" denotes the violation of the value function constraint for the iterates generated by iP-DwCA, i.e., $(f(\mathbf{x}^k, \mathbf{y}^k) - v_\gamma (\mathbf{x}^k, \mathbf{y}^k))/|I_{\text{tr}}|$ for iP-DwCA. The test (hold-out) error rate was calculated as follows:
			\[
			\text{Test Err. Rate} = \frac1{|I_{\text{test}}|} \sum_{i\in I_{\text{test}}} \frac12 |\mathrm{sign}(\mathbf{w}^T \bm{\varphi}_{\sigma}(\mathbf{a}_i) -  c) - b_i|,    
			\] where 
			$\mathbf{w}$ and $c$ are defined as in \eqref{def_wb}. We use "Val. Err. Infeas." and "Test. Err. Rate Infeas." to denote the validation error and test error rate calculated directly using the lower-level variable values of iterates produced by iP-DwCA and VF-iDCA, respectively; see Section \ref{sec:exp-enet} for discussion of this point.

\begin{table}[ht]
	\centering 
	\caption{Results on the RBF Kernel SVM hyperparameter selection problem}
	\label{table:svm}
    \resizebox{\textwidth}{!}{
	\begin{tabular}{c cccccc}
		\toprule 
Method     &      Time(s) & Val. Err. & Test. Err. Rate & Val. Err. Infeas. & Test. Err. Rate Infeas. &  Feasibility\\
\midrule
iP-DwCA-E &   5.8$\pm$0.3 &  0.29$\pm$0.21&  0.15$\pm$0.06&  0.22$\pm$0.07&  0.14$\pm$0.02&   0.08$\pm$0.07  \\
iP-DwCA   &  30.3$\pm$24.3&  0.21$\pm$0.08&  0.14$\pm$0.02&  0.17$\pm$0.07&  0.14$\pm$0.02&   0.00$\pm$0.01  \\
Grid      &  35.3$\pm$1.0 &  0.31$\pm$0.06&  0.14$\pm$0.02& - &  - &  - \\
Random    &  35.8$\pm$1.0 &  0.30$\pm$0.06&  0.14$\pm$0.02& - &  - &  - \\
TPE       &  36.8$\pm$1.0 &  0.30$\pm$0.06&  0.14$\pm$0.02& - &  - &  - \\
		\bottomrule 
	\end{tabular}}
\end{table}

In Table \ref{table:svm}, the early-stopping version of iP-DwCA is significantly faster than grid search, random search, and TPE, while the non-early-stopping version is of comparable computational time. We see that the early-stopping version attains comparable validation and test error to grid search, random search, and TPE, and the non-early-stopping version attains lower validation error but similar test error to  grid search, random search, and TPE.

            Finally, we plotted the decision region defined by the initial point used in iP-DwCA and the output iterate generated by iP-DwCA for one trial in Figures \ref{fig:svm0} and \ref{fig:svm20}, respectively. From these figures, we observe that our algorithm can efficiently identify appropriate hyperparameters with good classification performance, even when the initial guess is coarse.

			\begin{figure}[h]
				\centering
				\includegraphics[width=.9\textwidth]{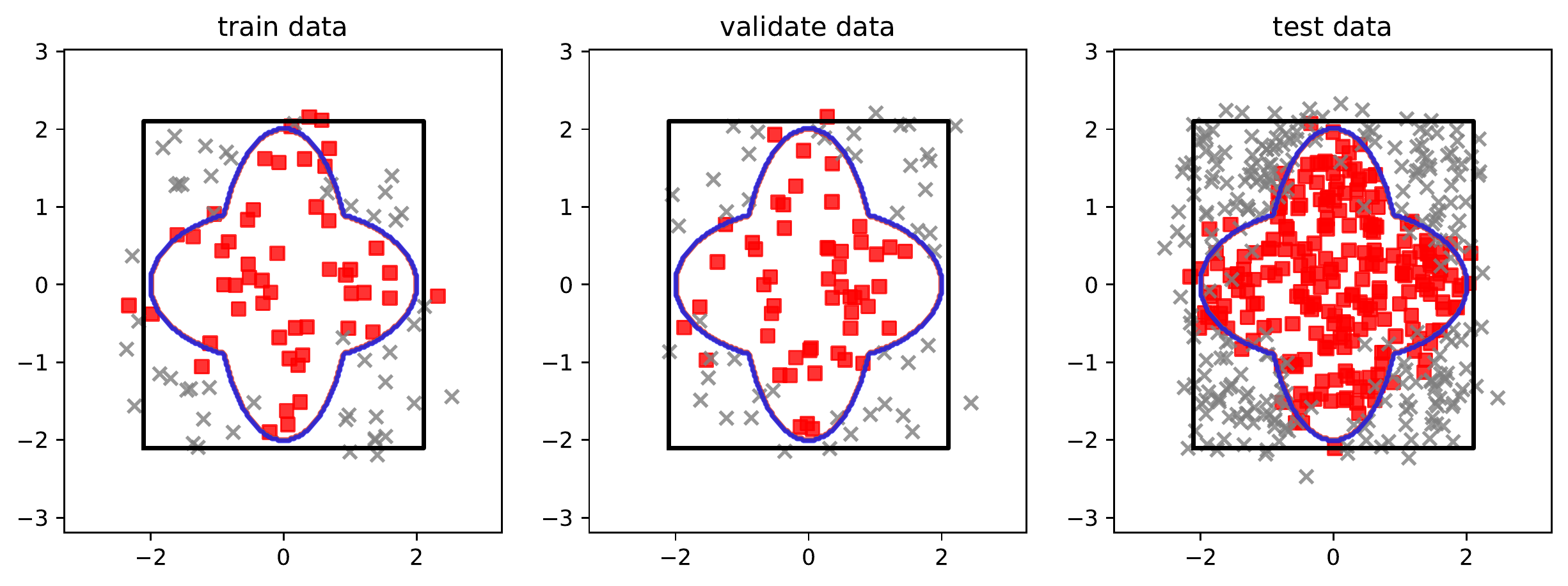}
				\caption{Generated synthetic data points}
				\label{fig:svm_ref}
			\end{figure}
			
			\begin{figure}[h]
				\centering
				\includegraphics[width=.9\textwidth]{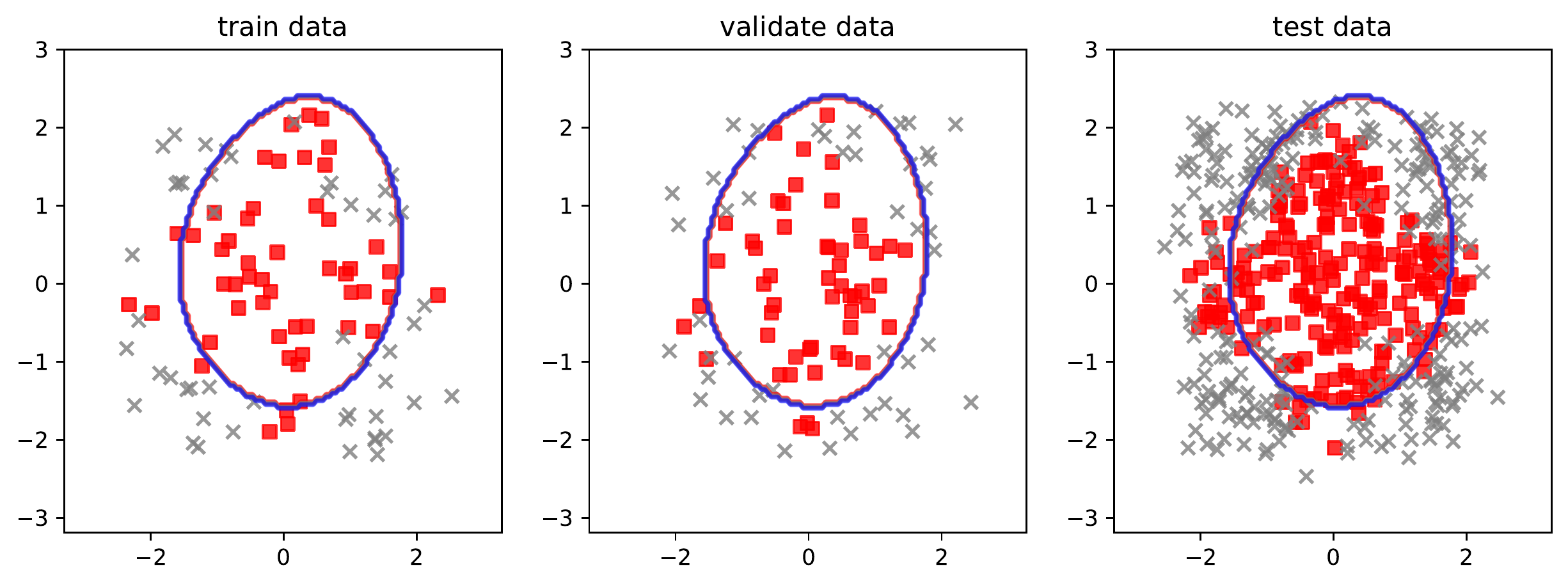}
				\caption{Decision region of initial point}
				\label{fig:svm0}
			\end{figure}
			\begin{figure}[h]
				\centering
				\includegraphics[width=.9\textwidth]{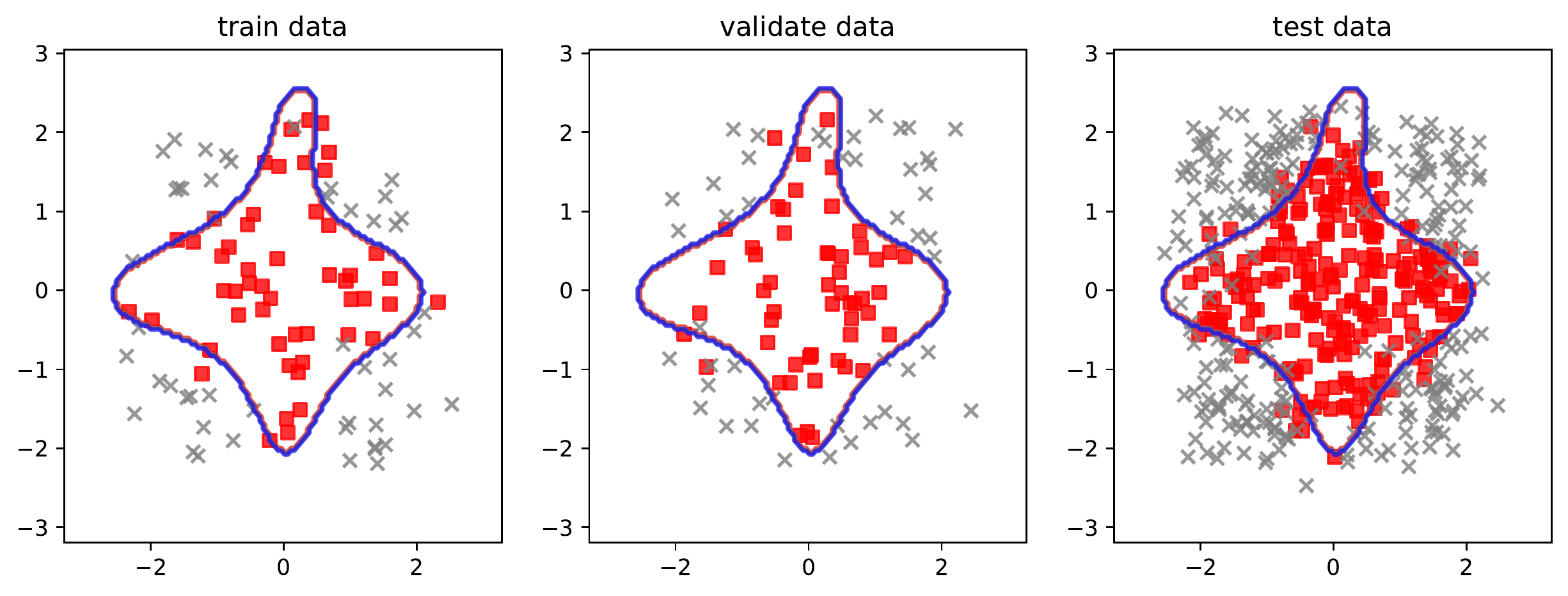}
				\caption{Decision region of output after 10 iterations.}
				\label{fig:svm20}
			\end{figure}

			\vskip 0.2in
			\bibliography{reference}

   			\newpage

			\appendix

\section{Proof of Lemma \ref{suff_decreasenew}} \label{appendix_A}

		By using the same arguments as in the proof of Lemma 1 in \cite{ye2021difference}, we can have following decrease result,
			\begin{equation}\label{suff_decrease_proof_eq1}
				\begin{aligned}
					\varphi_k(z^{k+1}) \le \varphi_k(z^{k}) + \frac{\alpha}{4}\|z^{k} - z^{k-1}\|^2 .
				\end{aligned}
			\end{equation}
			Next, since  $h_i^*$ is the conjugate function of $h_i$, $i = 0, 1$,  by definition
			\begin{equation*}
				-h_i(z^k) \le  -\langle \xi_i^{k-1}, z^k \rangle + h_i^*( \xi_i^{k-1}), \quad i = 0,1,\label{conjug}
			\end{equation*}
			and thus by taking into account the fact that $z^k\in \Sigma$, we have
			\begin{equation}\label{suff_decrease_proof_eq2}
				\begin{aligned}
					\varphi_k(z^{k}) &= g_0(z^k) - h_0(z^k) +\beta_k \max\{g_1(z^k) -h_1(z^k), 0\} \\
					& \le g_0(z^k) - \langle \xi_0^{k-1}, z^k \rangle + h_0^*( \xi_0^{k-1}) +\beta_k \max\{g_1(z^k) -\langle \xi_1^{k-1}, z^k \rangle + h_1^*( \xi_1^{k-1}) , 0\} \\
					& = E_{\beta_k}(z^k,z^{k-1},\xi^{k-1}) - \frac{\alpha}{4}\|z^k - z^{k-1}\|^2
				\end{aligned}
			\end{equation}
			Again, as $h_i^*$ is the conjugate function of $h_i$, and $\xi_i^k \in \partial h_i(z^k)$, $i = 0, 1$, there holds that
			\begin{equation*}
				-h_i(z^k) + \langle \xi_i^{k}, z^k \rangle = h_i^*(\xi_i^{k}), \quad i = 0, 1,
			\end{equation*}
			which implies
			\begin{equation*}\label{suff_decrease_proof_eq3}
				\begin{aligned}
					{\varphi}_k(z^{k+1}) = \, & g_0(z^{k+1}) - h_0(z^k) -\langle \xi_0^k, {z^{k+1}-z^k}\rangle \\
					&+\beta_k \max\{g_1(z^{k+1}) -h_1(z^k)- \langle\xi_1^k, z^{k+1}-z^k\rangle, 0\} +\frac{\alpha}{2} \|z^{k+1}-z^k\|^2 \\
					= \,&  g_0(z^{k+1}) -\langle \xi_0^k, z^{k+1}\rangle + h^*_0(\xi_0^k) \\
					&+\beta_k \max\{g_1(z^{k+1}) - \langle\xi_1^k, z^{k+1}\rangle + h^*_1(\xi_1^k), 0\} +\frac{\alpha}{2} \|z^{k+1}-z^k\|^2 \\
					= \, &E_{\beta_k}(z^{k+1},z^{k},\xi^{k}) + \frac{\alpha}{4} \|z^{k+1} - z^k\|^2.
				\end{aligned}
			\end{equation*}
			Combining with \eqref{suff_decrease_proof_eq1} and \eqref{suff_decrease_proof_eq2} gives us \eqref{suff_decrease_eq2}.
			
			It remains to prove \eqref{re_err_eq}. For convenience, we denote by $f_{1k}(z) := \max\{g_1(z) -h_1(z^k)- \langle\xi_1^k, z-z^k\rangle, 0\}$ and $\hat{f}_1(z,\xi):= \max\{g_1(z) - \langle\xi_1, z\rangle + h_1^*(\xi_1), 0\}$. Then 
			\[
			\varphi_k(z) = g_0(z) - h_0(z^k) - \langle \xi_0^k, {z-z^k}\rangle +\beta_k f_{1k}(z) +\frac{\alpha}{2} \|z-z^k\|^2,
			\]
			and 
			\[
			E_{\beta} = g_0(z) - \langle \xi_0,z \rangle + h_0^*(\xi_0) + \delta_{\Sigma}(z) + \beta \hat{f}_1(z,\xi) + \frac{\alpha}{4} \|z-z_0\|^2.
			\]
			Since  $g_1(x) $ is convex and continuous  and thus regular,  by the calculus rule for the pointwise maximum (see e.g.  Proposition 2.3.12 in \cite{clarke1990optimization}),  we have that $f_{1k}(z)$ is regular and
			\begin{equation*}
				\begin{aligned}
					\partial f_{1k}(z)
					= \,&  \left  \{ \zeta ( \partial g_1(z) - \xi_1^k) 
					~\left|~  \zeta \in [0,1], \zeta = 1, \text{if} ~  g_1(z) -h_1(z^k)- \langle\xi_1^k, z-z^k\rangle > 0, \right. \right. \\
					& \hspace*{150pt} \left . \zeta = 0, \text{if} ~  g_1(z) -h_1(z^k)- \langle\xi_1^k, z-z^k\rangle < 0  \right \}.
				\end{aligned}\label{g1}
			\end{equation*}
			Similarly,  
			$\hat{f}_1(z,\xi)$ is regular and 
			\begin{equation*}
				\begin{aligned}
					\partial \hat{f}_1(z,\xi)  
					= \,&  \left \{ \zeta ( \left\{ \partial g_1(z) - \xi_1 \right \}\times \{0\} \times \left\{-z+ \partial h_1^*(\xi_1)\right\}) 
					~ \left|~ \zeta \in [0,1], \right. \right.\\
					& \hspace{20pt }\zeta = 1, \text{if} ~  g_1(z) - \langle\xi_1, z\rangle + h_1^*(\xi_1) > 0, 
					\left . \zeta = 0, \text{if} ~  g_1(z) - \langle\xi_1, z\rangle + h_1^*(\xi_1) < 0  \right \}.
				\end{aligned}\label{g2}
			\end{equation*}
			Since $h_1^*$ is the conjugate function of $h_1$, and $\xi_1^k \in \partial h_1(z^k)$, there holds that
			\begin{equation*}
				-h_1(z^k) + \langle \xi_1^{k}, z^k \rangle = h_1^*(\xi_1^{k}).
			\end{equation*}
			Hence 
			\[
			g_1(z^{k+1}) -h_1(z^k)- \langle\xi_1^k, z^{k+1}-z^k\rangle 
			=\,  g_1(x^{k+1}) - \langle \xi_1^k, z^{k+1}\rangle + h_1^*(\xi_1^k),
			\]
			and thus
			\[
			f_{1k}(z^{k+1}) = \hat{f}_1(z^{k+1},\xi^k),
			\]
			and 
			\begin{equation}\label{subdiff_equal}
				\partial  f_{1k}(z^{k+1})  = \Pi_z \partial \hat{f}_1(z^{k+1},\xi^k),
			\end{equation}
		where $\Pi_z$ denotes the operator of projecting on the $z$ component subspace.
			Since  $g_0(z) $ and $f_{1k}(z)$ are all convex and continuous  and thus regular,   according to the the subdifferential sum rule we have
			(see e.g. Proposition 2.3.3  in \citet{clarke1990optimization}),  we have
			$$\partial \varphi_k(z) =\partial g_0(z) - \xi_0^k+ \beta_k \partial f_{1k}(z) +{\alpha} (z-z^k) .$$
			Similarly we have
			\[
			\begin{aligned}
				\partial E_{\beta} (z, z_0, \xi) = &\left \{\left\{ \partial g_0(z) - \xi_0 + \mathcal{N}_\Sigma(z)+ \beta\zeta \left(\partial g_1(z) - \xi_1\right) + \frac{\alpha}{2}\left(z - z_0\right) \right \} \times \left\{ \frac{\alpha}{2}\left(z_0 - z\right) \right\} \right.\\
				&\left. \times \{-z + \partial h_0^*(\xi_0)\} \times \left\{\beta\zeta\left(-z+ \partial h_1^*(\xi_1) \right)\right\}
				~ \left|~ \zeta \in [0,1], \right. \right.\\
				& \hspace{20pt }\zeta = 1, \text{if} ~  g_1(z) - \langle\xi_1, z\rangle + h_1^*(\xi_1) > 0, 
				\left . \zeta = 0, \text{if} ~  g_1(z) - \langle\xi_1, z\rangle + h_1^*(\xi_1) < 0  \right \}. 
			\end{aligned}
			\]
			Hence by \eqref{subdiff_equal},  we have
			\begin{equation}
				\Pi_z	\partial E_{\beta_k}(z^{k+1},z^{k},\alpha^{k}) 
				= \partial \varphi_k(z^{k+1} )+ \mathcal{N}_\Sigma(z^{k+1}) - \frac{\alpha}{2}(z^{k+1}-z^k). \label{relation}
			\end{equation}
			
			Since $z^{k+1}$ is an approximate solution to problem \eqref{subp} satisfying inexact criteria \eqref{inexact2}, there exists a vector $e_k$ such that
			\[
			e_k \in \partial {\varphi}_k(z^{k+1}) + \mathcal{N}_\Sigma(z^{k+1})
			\]
			satisfying $$\|e_k\| \le \frac{\sqrt{2}}{2} \alpha\|z^k - z^{k-1}\|.$$
			It follows from \eqref{relation} that 
			$$e_k - \frac{\alpha}{2}(z^{k+1}-z^k)  \in \Pi_z \partial E_{\beta_k}(z^{k+1},z^{k},\xi^{k}).$$ Next,  using the fact that $\xi_i^k \in \partial h_i(z^{k})$ which is equivalent to $z^k\in \partial h_i^*(\xi_i^k)$ for $i = 0, 1$, we have
			\[ 
			\begin{aligned}
				&\partial E_{\beta_k}(z^{k+1},z^{k},\xi^{k}) 
				\ni \begin{pmatrix}
					e_k - \frac{\alpha}{2}(z^{k+1}-z^k)  \\
					\frac{\alpha}{2}(z^k- z^{k+1}) \\
					z^k - z^{k+1} \\
					\beta_k\zeta_{k+1}(z^{k} - z^{k+1})
				\end{pmatrix},
			\end{aligned}
			\]
			where $\zeta_{k+1} \in [0,1]$.
			\eqref{re_err_eq} then  follows immediately.

   \section{Proof of Theorem \ref{con_KL_alg1}}
    \label{appendix_B}

			By the assumption that the adaptive penalty sequence $\{\beta_k\}$ is bounded, $\beta_k = \bar{\beta}$ for all sufficiently large $k$ and we can assume without loss of generality that 
			$\beta_k = \bar{\beta}$ for all $k$.
			Since  $f_0(z)$ is assumed to be bounded below on $\Sigma$, $E_{\bar{\beta}}(z,z_0,\xi)$ is also bounded below. Then, according to Lemma \ref{suff_decreasenew}, we have $\lim_{k \rightarrow \infty} \|z^{k+1} - z^k\|^2 = 0$.
			
			Let $\Omega$ denote the set of all limit points of sequence $\{(z^k,z^{k-1},\xi^{k-1})\}$. Then we have that $\Omega$ is a closed set and $\lim_{k \rightarrow \infty} \mathrm{dist}((z^k,z^{k-1},\xi^{k-1}), \Omega) = 0$. Then, any accumulation point of the sequence $\{z^k\}$ corresponds to those in the set $\Omega$ is a KKT point of problem (DC).
			According to the assumptions that $h_0$, $h_1$ are locally Lipschitz on set $\Sigma$ and the sequence $\{z^k\}$ is bounded, and $\xi_i^{k-1} \in \partial h_i(z^{k-1})$, $i = 0, 1$, we get the boundedness of the sequence $\{\xi^k\}$ and thus $\Omega$ is a compact set.
			
			Since the function $E_{\bar{\beta}}(z,z_0,\xi)$ is bounded below and continuous on $\Sigma$, $\lim_{k \rightarrow \infty} \|z^{k+1} - z^k\| = 0$, and by Lemma \ref{suff_decreasenew} that $E_{\bar{\beta}}(z^k,z^{k-1},\xi^{k-1})$ is decreasing as $k$ increases , we have that there exists a constant $\bar{E}$ such that for any subsequence $\{(z^l,z^{l-1},\xi^l)\}$ of sequence $\{(z^k,z^{k-1},\xi^{k-1})\}$,
			\[
			\bar{E} = \lim_{l \rightarrow \infty} E_{\bar{\beta}}(z^l,z^{l-1},\xi^l) = \lim_{k \rightarrow \infty} E_{\bar{\beta}}(z^k,z^{k-1},\xi^k),
			\]
			and thus the function $E_{\bar{\beta}}(z,z_0,\xi)$ is constant on $\Omega$.
			We can assume that $E_{\bar{\beta}}(z^k,z^{k-1},\xi^k) > \bar{E}$ for all $k$. Otherwise, if there exists $k > 0$ such that $E_{\bar{\beta}}(z^k,z^{k-1},\xi^k) = \bar{E}$, then by the above equation and Lemma \ref{suff_decreasenew}, we get $\|z^{k+1} - z^k\| = 0$ when $k$ is sufficiently large, which implies the convergence of sequence $\{z^k\}$ and the conclusion follows immediately. 
			
			Since $\lim_{k \rightarrow \infty} \mathrm{dist}((z^k,z^{k-1},\xi^{k-1}), \Omega) = 0$ and $\lim_{k \rightarrow \infty} E_{\bar{\beta}}(z^k,z^{k-1},\xi^k) = \bar{E}$, for any $\epsilon, \eta > 0$, there exists $k_0$ such that $\mathrm{dist}((z^k,z^{k-1},\xi^{k-1}), \Omega) < \epsilon$ and $\bar{E} < E_{\bar{\beta}}(z^k,z^{k-1},\zeta^k) < \bar{E} + \eta$ for $k \ge k_0$.	
			As $E_{\bar{\beta}}(z,z_0,\xi)$ satisfies the Kurdyka-\L{}ojasiewicz property at each point in $\Omega$, and $E_{\bar{\beta}}(z,z_0,\xi)$ is equal to a finite constant on $\Omega$, we can apply Lemma \ref{uniformKL} to obtain a continuous concave function $\phi$ such that for any $k \ge k_0$,
			\[
			\phi' \left(E_{\bar{\beta}}(z^k,z^{k-1},\xi^{k-1}) - \bar{E} \right) \mathrm{dist} \left( 0, \partial E_{\bar{\beta}}(z^k,z^{k-1},\xi^{k-1}) \right) 	\ge 1.
			\]
			Combining with \eqref{re_err_eq} yields
			\[
			\phi' \left (E_{\bar{\beta}}(z^k,z^{k-1},\xi^{k-1}) - \bar{E}\right ) 
			\cdot\,  \left(\frac{\sqrt{2}}{2} \alpha\|z^{k-1} - z^{k-2}\| +  (\bar{\beta}+\alpha+1) \|z^{k} - z^{k-1}\| \right)\ge 1.
			\]
			The concavity of $\phi$ and \eqref{suff_decrease_eq2} implies 
			\[
			\begin{aligned}
				&\phi'( E_{\bar{\beta}}(z^k,z^{k-1},\xi^{k-1}) - \bar{E}) \cdot \frac{\alpha}{4} \|z^{k+1} - z^k\|^2 \\
				\le \ &\phi'( E_{\bar{\beta}}(z^k,z^{k-1}, \xi^{k-1}) - \bar{E}) \cdot \left( E_{\bar{\beta}}(z^k,z^{k-1},\xi^{k-1})  - E_{\bar{\beta}}(z^{k+1},z^{k},\xi^{k}) \right) \\
				\le \ &\phi\left( E_{\bar{\beta}}(z^k,z^{k-1},\alpha^{k-1}) - \bar{E} \right) -  \phi\left( E_{\bar{\beta}}(z^{k+1},z^{k},\alpha^{k}) - \bar{E} \right).
			\end{aligned}
			\]
			Combining the above two inequalities, we obtain 
			\[
			\begin{aligned}
				&\frac{4}{\alpha}\left(\frac{\sqrt{2}}{2} \alpha\|z^{k-1} - z^{k-2}\| +  (\bar{\beta}+\alpha+1) \|z^{k} - z^{k-1}\| \right) \\
				&\cdot \left[\phi\left( E_{\bar{\beta}}(z^k,z^{k-1}, \xi^{k-1}) - \bar{E} \right) -  \phi\left( E_{\bar{\beta}}(z^{k+1},z^{k}, \xi^{k}) - \bar{E} \right) \right] 
				\ge\, \|z^{k+1} - z^k\|^2.
			\end{aligned}
			\]
			Multiplying both sides of this inequality by 4 and taking the square root, then by $2ab \le a^2 + b^2$, we have
			\[
			\begin{aligned}
				4\|z^{k+1} - z^k\| \le \,&\|z^{k} - z^{k-1}\| + \|z^{k-1} - z^{k-2}\|\\  &+ \frac{16(\bar{\beta}+\alpha+1)}{\alpha}\left[\phi\left( E_{\bar{\beta}}(z^k,z^{k-1},\xi^{k-1}) - \bar{E} \right) -  \phi\left( E_{\bar{\beta}}(z^{k+1},z^{k},\xi^{k}) - \bar{E} \right) \right].
			\end{aligned}
			\]
			Summing up the above inequality for $i = k_0 + 1, \ldots, k$, we have
			\[
			\begin{aligned}
				\sum_{i = k_0}^k 4\|z^{i+1} - z^i\| \le  &	\sum_{i = k_0}^k \left( \|z^{i} - z^{i-1}\| + \|z^{i-1} - z^{i-2}\| \right)\\
				& + \frac{16(\bar{\beta}+\alpha+1)}{\alpha} \left[\phi\left( E_{\bar{\beta}}(z^{k_0},z^{k_0-1},\xi^{k_0-1}) - \bar{E} \right) - \phi\left( E_{\bar{\beta}}(z^{k+1},z^{k},\xi^{k}) - \bar{E} \right) \right],
			\end{aligned}
			\]
			and since $\phi \ge 0$, we get
			\[
			\begin{aligned}
				\sum_{i = k_0}^k 2\|z^{i+1} - z^i\| \le  &2\|z^{k_0} - z^{k_0-1}\| + \|z^{k_0-1} - z^{k_0-2}\| + \frac{16(\bar{\beta}+\alpha+1)}{\alpha}\phi\left( E_{\bar{\beta}}(z^{k_0},z^{k_0-1},\xi^{k_0-1}) - \bar{E} \right).
			\end{aligned}
			\]
			Taking $k \rightarrow \infty$ in the above inequality shows that
			\[
			\sum_{k = 1}^\infty \|z^{k+1} - z^k\| < \infty,
			\]
			and thus the sequence $\{z^k\}$ is a Cauchy sequence. Hence the sequence $\{z^k\}$ is convergent and we get the conclusion.

		\end{document}